\documentclass[10pt,fleqn]{article}
\usepackage{mathtools,amsfonts,amsmath, amsthm,amssymb}

\usepackage{csquotes}
\usepackage{pdfsync}
\usepackage[active]{srcltx}

\usepackage[utf8]{inputenc}
\usepackage[T1]{fontenc}
 \usepackage{fouriernc}
 \usepackage{latexsym}
\usepackage{esint}

\usepackage[margin=1in]{geometry} 
 \swapnumbers
 \numberwithin{equation}{section}

\usepackage{color}
 \definecolor{db}{rgb}{0.0,0.0,0.8} 
\definecolor{dg}{rgb}{0.0,0.55,0.14}
\definecolor{dr}{rgb}{0.5,0,0.07}

  \usepackage{hyperref}
 \hypersetup{frenchlinks=true,colorlinks=true,linkcolor=db,citecolor=dg,
 bookmarks=true}

 \usepackage{enumerate}

\newtheorem{theorem}{Theorem}[section]
\newtheorem{proposition}[theorem]{Proposition}
\newtheorem{lemma}[theorem]{Lemma}
\newtheorem{corollary}[theorem]{Corollary}
\theoremstyle{definition}

\theoremstyle{definition}

\theoremstyle{definition}

\theoremstyle{definition}

\theoremstyle{definition}
\newtheorem{definition}[theorem]{Definition}
\theoremstyle{definition}
\newtheorem{remark}[theorem]{Remark}
\theoremstyle{definition}

\newcounter{step}

\newcommand{\rlemma}[1]{Lemma~\ref{#1}}
\newcommand{\rth}[1]{Theorem~\ref{#1}}

\newcommand{\rprop}[1]{Proposition~\ref{#1}}
\newcommand{\rrem}[1]{Remark~\ref{#1}}
\newcommand{\rsec}[1]{Section~\ref{#1}}

\def\be{\begin{equation}}
\def\ee{\end{equation}}
\def\bes{\begin{equation*}}
\def\ees{\end{equation*}}
\def\bt{\begin{theorem}}
\def\et{\end{theorem}}
\def\bpr{\begin{proposition}}
\def\epr{\end{proposition}}
\def\bl{\begin{lemma}}
\def\el{\end{lemma}}
\def\bc{\begin{corollary}}
\def\ec{\end{corollary}}
\def\br{\begin{remark}}
\def\er{\end{remark}}
\def\ben{\begin{enumerate}}
\def\bena{\begin{enumerate}[a)]}
\def\een{\end{enumerate}}
\def\bit{\begin{itemize}}
\def\iit{\end{itemize}}

\def\supp{\operatorname{supp}}
\def\dist{\operatorname{dist}}

\def\deg{\operatorname{deg}}
\def\div{\operatorname{div}}

\def\sgn{\operatorname{sgn}}
\def\pv{\operatorname{p. v. }}

 \newcommand{\Prod}{\mathop{\prod}\limits}
\DeclareMathAlphabet{\mathonebb}{U}{bbold}{m}{n}

\def\R{{\mathbb R}}

\def\Z{{\mathbb Z}}

\def\fo{\forall\, }

\def\va{\varphi}
\def\ue{u_\varepsilon}
\def\d{\displaystyle}
\def\im{\imath}
\def\ve{\varepsilon}
\def\p{\partial}
\def\l{\label}
\def\O{\Omega}
\def\po{\partial\Omega}

\def\na{\nabla}

\def\so{{\mathbb S}^1}

\date{\today}
\title{Asymptotic behavior of critical points of an energy involving a
  loop-well potential}
\author{Petru Mironescu    \thanks{Universit\'e de Lyon,  CNRS UMR 5208, Universit\'e Lyon 1, Institut Camille Jordan, 43 blvd. du 11 novembre 1918, F-69622 Villeurbanne cedex, France. Email address: mironescu$@$math.univ-lyon1.fr}         \and
Itai Shafrir \thanks{Department of Mathematics, Technion - I.I.T., 32000 Haifa, Israel. Email address: shafrir$@$math.technion.ac.il}
}

\begin{document}

\maketitle

%
\begin{abstract}
We describe the asymptotic behavior of critical points of $\int_\O [(1/2)|\na u|^2+W(u)/\ve^2]$ when $\ve\to 0$. Here, $W$ is a Ginzburg-Landau type potential, vanishing on a simple closed curve $\Gamma$.  Unlike the case of the standard Ginzburg-Landau potential $W(u)=(1-|u|^2)^2/4$, studied by Bethuel, Brezis and H\'elein, we do not assume any symmetry on $W$ or $\Gamma$. In order to overcome the difficulties due to the lack of symmetry, we develop new tools which might be of independent interest. 
\end{abstract}

\section{Statement of the problem}

Let $\Omega\subset\R^2$ be a smooth bounded star-shaped domain. Let $\Gamma\subset\R^2$ be a smooth simple curve and let $g:\partial\Omega\to\Gamma$ be a smooth boundary datum of degree $d$. Consider, for every $\ve>0$, a critical point  $u_\ve\in H^1_g(\Omega ;\R^2)$ of the energy 
 \begin{equation}
\label{eq:90}
 E_\ve(u)=\int_\Omega\left[\frac{1}{2}|\nabla u|^2+\frac{W(u)}{\varepsilon^2}\right].
 \end{equation}
 
 Here, $W:\R^2\to [0,\infty)$ is a smooth potential vanishing precisely on $\Gamma$; for the exact assumptions on $W$, see \eqref{h1}--\eqref{h5} below.

In the Ginzburg-Landau (GL) case, i.e., when $W(u)=(1-|u|^2)^2/4$, the
asymptotic behavior of $\{\ue\}$ when $\varepsilon\to 0$ was studied by
Bethuel, Brezis and H\'elein, first for minimizers when  the
boundary condition has zero degree in
\cite{bbh93}, and later for minimizers and, more generally, for critical points  
for arbitrary  boundary datum in the seminal work \cite{bbh94}. 
 
The analysis in \cite{bbh94} for {\em minimizers} of the
 GL energy can be adapted 
 with no significant difficulty to the case of general $W$, at least
 when $W$ is {\em non-degenerate}, see \eqref{eq:117}. Using more involved arguments, it is even possible to describe the asymptotic behavior of minimizers in the case of a general 
 boundary condition $g$
 that does not necessarily take values into $\Gamma$; see    Andr\'e and Shafrir \cite{ansh07}.

 We address here the question of the asymptotic behavior of  critical points of the
 energy \eqref{eq:90}, i.e., of solutions of 
 \begin{equation}
 \label{eq:EL}
 \begin{cases}
 \Delta u_\ve=\d\frac{1}{\varepsilon^2}\nabla W(u_\ve)&\text {in }\Omega\\
 u_\ve=g&\text{on }\partial\Omega
 \end{cases}
 \end{equation}
that need not be energy minimizing with respect to their own boundary condition. As we will see below, the answer to this question requires new ideas and ingredients.  
We emphasize that the starshapeness condition on  $\Omega$ is
crucial to our analysis, as it was in \cite[Chapter~\rm{X}]{bbh94}.
 As far as we know the problem about
critical points in a general simply connected domain is still open
even in the case of the usual Ginzburg-Landau potential.

The method of proof in \cite[Chapter~\rm{X}]{bbh94} for critical
 points of the GL energy is based on a clever decomposition of the 
 gradient $\nabla\ue$. Its starting point is the identity
 \begin{equation}
\label{eq:108}
\frac{\partial}{\partial x_1} \left(\ue\times\frac{\p}{\p x_1}\ue\right) +  \frac{\partial}{\partial x_2} \left(\ue\times\frac{\p}{\p x_2}\ue\right)=0, 
\end{equation}
 which is a direct consequence of the fact that $W(u)=W(|u|)$ in the
 GL case. We could not find an analogous identity to \eqref{eq:108} for
 general $W$. Our method is different and relies on two main tools: 
\ben
\item
 Selection of \enquote{good rays} (see
 Subsection~\ref{subsec:log}).
 \item
A maximum principle for the phase
 (see \rprop{prop:max}). 
 \een
  Combined, they allow us to prove a crucial estimate, namely 
 \begin{equation}
\label{eq:118}
   E_\varepsilon(\ue)\leq C(|\log\varepsilon|+1).
 \end{equation}

 The first ingredient is new even for the GL energy (and leads to a simplification of the original arguments in \cite[Chapter~\rm{X}]{bbh94}), and the second one is much more subtle in the case of a general potential $W$ than in the GL case. 

\smallskip
For the analysis of solutions to \eqref{eq:EL} we will need, in the spirit of  \cite{bbh94}, the additional assumption that $\Omega$ is strictly star-shaped. This assumption enables us to prove that the second
 term in the energy \eqref{eq:90} remains bounded when
 $\varepsilon\to0$, and then to perform the \enquote{bad discs} construction \`a la
 Bethuel-Brezis-H\'elein~\cite{bbh94}, which is the starting point of the study of the
 location of the vortices.
 
  The remaining part of the analysis is similar to the one in \cite{bbh94} (with some technical complications), and leads to our main result, Theorem \ref{th:main} below. In order to state it, we first present all the assumptions on $\Gamma$ and $W$.
\be
\l{h1}
W:\R^2\to[0,\infty)\text{ is a smooth function satisfying }W^{-1}(\{0\})=\Gamma
\ee
and
\be
\l{h2}
\Gamma\text{ is a simple closed smooth curve in }\R^2.  
\ee
We assume without loss
of generality that 
\be
\l{h3}
|\Gamma|=2\pi
\ee
 and consider
 \be
 \l{h4} 
 \tau:\so\to\Gamma\text{ an arc length parametrization of }\Gamma.
 \ee
 We also suppose that $W$ is {\em non-degenerate} in the following sense:
 \begin{equation}
   \label{eq:117}
   W(\zeta)\geq \mu\dist^2(\zeta,\Gamma)\text{ if }\dist(\zeta,\Gamma)<\delta, 
 \end{equation}
for some $\mu,\delta>0$ (and then it follows from \eqref{h1} that \eqref{eq:117} holds
on any compact subset of $\R^2$). 

   In addition, we impose the following coercivity assumption on the behavior of $W$ at infinity:
 \begin{equation}
\label{h5}
\frac{\partial W}{\partial r}(z)\geq 0\text{ for } |z|=r>R_0,
\end{equation}
 for some $R_0>\max\{|z|;\,z\in\Gamma\}$.
 
\begin{theorem}
\label{th:main}
 Let $\Omega$ be a smooth, bounded, strictly star-shaped domain in $\R^2$.
 Let $W$, $\Gamma$ and $\tau$ satisfy \eqref{h1}--\eqref{h5}. Let $g:\partial\Omega\to\Gamma$ be a smooth
 boundary condition of degree $d$. For each $\varepsilon>0$, let $u_\varepsilon$ denote a
 solution of \eqref{eq:EL}. Then up to a subsequence we have 
\begin{equation}
\label{eq:91}
 u_{\varepsilon_n}\to
  u_*=\tau\left(e^{\im\eta(z)}\left(\frac{z-a_1}{|z-a_1|}\right)^{D_1}\cdots\left(\frac{z-a_N}{|z-a_N|}\right)^{D_N}\right) \text{ in }C^{1,\alpha}\left(\overline\Omega\setminus \{a_1,\ldots, a_N\}\right),
\end{equation}
where 
\ben
\item
$a_1,\ldots,a_N\in\Omega$ are mutually distinct points.
\item
 $D_1,\ldots, D_N\in\Z\setminus\{0\}$ satisfy the compatibility condition $\sum_{j=1}^N D_j=d$.
\item
$\eta$ is a harmonic function in $\overline\Omega$. 
\item
$\alpha\in (0,1)$.
\een
\end{theorem}

\smallskip
\par In the spirit of \cite{bbh94}, we may also prove that the
configuration $(a_1,\ldots,a_N)$ is a critical point
of a suitable renormalized energy associated with the degrees
$(D_j)_{j=1}^N$ and the boundary condition; see Remark \ref{rem:ren} in Section \ref{sec:nonzerodeg}.

Let us mention that non minimizing solutions do exist. For the GL energy, their existence  was
  established in different situations. In the special case where $\O$ is the unit disc and 
  $g(z)=z^d$, with $|d|\geq2$, the GL energy has critical points of the form $\ue(r e^{\im\theta})=f_\varepsilon(r)\,e^{\im d\theta}$, and these solutions are not
  minimizing for sufficiently small $\varepsilon$  \cite{bbh94}.  Non
  minimizing critical points also exist when $d=0$:   F.H.~Lin~\cite{lin95}
  constructed examples of 
\enquote{mixed vortex-antivortex solutions}. More specifically, 
for all $ N\geq 1$ there exists $g_N:\partial\Omega\to \so$ of
degree $0$ and non minimizing corresponding critical points $u_{\varepsilon_n}$ such that
\begin{equation*}
 u_{\varepsilon_n}\to u_*=e^{\im\eta_N(z)}\, \Prod_{j=1}^{2N}\left(\frac{z-a_{j, N}}{|z-a_{j, N}|}\right)^{(-1)^{j-1}}.
\end{equation*}

 Other existence results concerning non minimizing solutions for the 
 the GL energy were proved
 by Almeida and Bethuel~\cite{almeida-beth98} and  by F.~Zhou and
 Q.~Zhou~\cite{zhou-qing99}, using variational and topological methods. We believe that at least some of these
 methods  lead to the existence of non minimizing critical points of \eqref{eq:EL} for a general $W$,   but we did not investigate
 this issue.

Except for the upper bound \eqref{eq:118}, we did not establish a more
precise estimate for the energy $E_\varepsilon(\ue)$. In the case of
the GL-energy, Comte and Mironescu \cite{cm-rem} proved  that the
following is true:
\begin{equation}
  \label{eq:119}
  E_\varepsilon(\ue)=\pi\left(\sum_{j=1}^N D_j^2\right)|\log\ve|+O(1).
\end{equation}
It would be interesting to generalize the validity of \eqref{eq:119}
to our setting.

\medskip
 The paper is organized as follows. In \rsec{sec:prelim} we introduce some
notation and prove a maximum principle for the phase, that plays an
important role in the remaining part of the paper. In 
\rsec{sec:no-vortices} we study the case of boundary data of zero
degree ($d=0$) under the additional assumption that the solutions stay
close to $\Gamma$, i.e., no vortices appear. The techniques of this
section are used in \rsec{sec:dep-eps} to treat the more general case
of a  boundary data
 depending on $\varepsilon$ (again, for vortex-less solutions). This
 latter case is very useful in the proof of convergence away from the
 vortices in
 \rth{th:main}. \rsec{sec:nonzerodeg} is devoted to the proof
 of the main result, \rth{th:main}.

\subsubsection*{Acknowledgments.} The first author (PM) was partially  supported  by the LABEX MILYON (ANR-10-LABX-0070) of Universit\'e de Lyon,
within the program \enquote{Investissements d'Avenir} (ANR-11-IDEX-0007) operated
by the French National Research Agency (ANR). The second author (IS)
was supported by  the Israel Science Foundation (Grant No. 999/13). Part of this work was done while  IS was visiting the University Claude Bernard Lyon 1. He thanks the Mathematics Department for its hospitality.

\tableofcontents
\section{Preliminaries}
\label{sec:prelim}
 \subsection{Coordinates and Euler-Lagrange equations}
 Consider $\Gamma_{\delta}:=\{z\in \R^2;\,\text{dist}(z,\Gamma)<\delta\}$. For sufficiently small $\delta_\Gamma$ (depending on $\Gamma$)  the Euclidean nearest point projection $\Pi$ on $\Gamma$ is well-defined and  smooth in $\Gamma_{\delta_\Gamma}$  (see e.g. \cite[Sec.~14.6]{gitr01}). 
 
 Assume in what follows that $u:\omega\to\R^2$ is a smooth map such that 
 \be
 \l{h9}
 u(x)\in \Gamma_{\delta_\Gamma},\ \fo x\in\omega.
 \ee
 (Here, $\omega\subset\R^2$ is some open set.) 
Locally in $\omega$, we can associate to $u$ two smooth coordinates, $t$ and $\va$, such that $\Pi\circ u=\tau \left(e^{\im\va}\right)$ and $t$ is the signed distance of $u$ to $\Gamma$ (taken with the plus sign inside $\Gamma$). Analytically, this means that the functions $t$ and $\va$ satisfy
$(t(x),\va(x))\in(-\delta_\Gamma,\delta_\Gamma)\times\R$ and
 \begin{equation}
 \label{eq:utotphi}
 u(x)=\tau\left(e^{\im\va (x)}\right)+t(x)\,\vec{n}\left(\tau\left(e^{\im\va (x)}\right)\right).
\end{equation}
 Here, $\vec{n}(z)$ denotes the inward  unit normal to $\Gamma$ at
 the point $z\in\Gamma$. 
 
 Equivalently, we have 
 \be
 \l{ha1}
 \Pi (u(x))=\tau\left(e^{\im\va (x)}\right)\text{ and }t(x)= (u(x)-\Pi (u(x)))\cdot \vec n\, (\Pi(u(x))).
 \ee

 Note that $t$ is globally defined, but $\va$ is  only {\em locally} defined in $\omega$, and that $\va$  is (locally) unique mod $2\pi$. It is useful to note that $\va$ is {\it globally} defined when $\omega$ is simply connected.
 
 A simple calculation (see \cite[Lemma 4.1]{as03}) shows that for $u$ satisfying \eqref{h9} we have (denoting by $\kappa(z)$ the curvature of $\Gamma$ at the point $z\in\Gamma$)
 \begin{equation}
 \l{ha2}
|\nabla u|^2=\left(1-t\,\kappa\left(\tau(e^{\im\va})\right)\right)^2|\nabla\va|^2+|\nabla t|^2=\left(1-t\,\kappa\left(\Pi\circ u\right)\right)^2|\nabla\va|^2+|\nabla t|^2.
\end{equation}
Moreover, for such $u$ we have (using \eqref{eq:117}) that
\be
\l{ha3}
W(u)=\alpha(\va,t)\, t^2
\end{equation}
 where $\alpha(\va,t)$ is a smooth positive  function,
 $2\pi$-periodic in the $\va$-variable.

 Assume next that $u=\ue$ is a solution of \eqref{eq:EL} in $\O$ and that $\omega\subset\O$ is such that \eqref{h9} holds.  Then locally in $\omega$ we
 may use \eqref{ha3} to write the  Euler-Lagrange equations \eqref{eq:EL} for the function $u$ in the new coordinates $t$ and $\va$ as  follows.
\begin{subequations} 
\label{eq:ELN}
\begin{gather}
-\div (a\nabla\va)=b|\nabla\va|^2-\frac{\alpha_\va t^2}{\ve^2}\label{eq:ELN-phi},\\
-\Delta t+(2\alpha+\alpha_t t)\frac{t}{\varepsilon^2}=c|\nabla\varphi|^2.\label{eq:ELN-t}
\end{gather}
\end{subequations}

In \eqref{eq:ELN}, the coefficients $a=a(\va, t)$, $b=b(\va, t)$ and $c=c(\va, t)$ are given by
\begin{equation}
\label{eq:abc}
\left\{
\begin{aligned}
&a=\left(1-t\kappa\left(\tau(e^{\im \va})\right)\right)^2=1+O(t)\\
&b=-\frac{1}{2}a_\va=O(t)\\
&c=-\frac{1}{2}a_t=O(1)
\end{aligned}.
\right.
\end{equation}

\subsection{A maximum principle for the phase}
 By \eqref{ha3}--\eqref{eq:abc}, for sufficiently small $\delta_0\in(0,\delta_\Gamma)$
 there exist positive
 constants $c_0,\ldots,c_5$ such that for
 $|t|\leq \delta_0$ there holds:
\begin{subequations}
\label{eq:1}
 \begin{gather}
|1-a|\leq c_0|t|,\label{eq:1-0}\\
 2\alpha-|\alpha_t t|\geq c_1,\label{eq:1-1}\\
|c|\leq c_2,\label{eq:1-2}\\
\left|\frac{b}{a}\right|\leq c_3|t|,\label{eq:1-3}\\
\left|\frac{a_t}{a}\right|\leq 2c_4,\label{eq:1-4}\\
\left|\frac{\alpha_\va}{a}\right|\leq c_5.\label{eq:1-5}
\end{gather}
\end{subequations}

Note that $\delta_0$ depends only on $\Gamma$.

Next we prove a maximum principle for the phase $\varphi$, that
will be useful throughout the paper. For this purpose, we introduce two numbers,  $0<\delta_1<\delta_0$ and
$m>0$,  satisfying  
\begin{equation}
  \label{eq:2}
  \frac{c_5}{c_1}\leq m
\end{equation}
 and
 \begin{equation}
   \label{eq:11}2c_4\delta_1+m(mc_2+c_3)\delta_1^3<1.
 \end{equation}
 
 Note that $\delta_1$ and $m$ depend only on  $\Gamma$ and $W$.

\begin{proposition}
\label{prop:max}
Let
 $u=u_\ve$ be a critical point of $E_\ve$ in  a   bounded simply connected domain $\omega$,  continuous on $\overline \omega$ and
satisfying $\dist (u(x), \Gamma)\le \delta_1$, $\fo x\in\overline \omega$. Consider 
$t=t_\ve,\varphi=\varphi_\varepsilon$ associated to $u$ via \eqref{eq:utotphi}. Then
\begin{subequations}
\label{eq:max-pr}
\begin{gather}
\min_{\overline \omega} \left(\va-\frac{mt^2}{2}\right)=\min_{\partial \omega} \left(\va-\frac{mt^2}{2}\right),\label{eq:max}\\
\max_{\overline \omega} \left(\va+\frac{mt^2}{2}\right)=\max_{\partial \omega} \left(\va+\frac{mt^2}{2}\right).\label{eq:min}
\end{gather}
\end{subequations}
\end{proposition}

\begin{corollary}
\label{rem:max}
If, in addition to the hypotheses of Proposition  \ref{prop:max}, we have $u_\ve(\partial \omega)\subseteq \Gamma$, then
\be
\l{ha4}
\min_{\partial \omega}\va\leq \va(x)-\frac{mt^2(x)}{2}\leq
\va(x)+\frac{mt^2(x)}{2} \leq \max_{\partial \omega}\va,\quad\fo x\in \omega.
\ee
In particular, 
\begin{equation}
\label{eq:3}
  \min_{\partial \omega}\va\leq\va(x)\leq
    \max_{\partial \omega}\va,\ \forall\, x\in\overline \omega.
\end{equation}
\end{corollary}

\begin{proof}[Proof of Proposition  \ref{prop:max}]
  First, we may rewrite the equation \eqref{eq:ELN-phi} as
  \begin{equation*}
    -\Delta\varphi=\frac{1}{a}\nabla a\cdot\nabla\varphi +\frac{b}{a}|\nabla\varphi|^2-\frac{\alpha_\varphi t^2}{a\ve^2}.
  \end{equation*}
 Using 
 \begin{equation*}
    \nabla a\cdot\nabla\varphi=a_\varphi |\nabla\varphi|^2+a_t\nabla\varphi\cdot\nabla t=-2b|\nabla\varphi|^2+a_t\nabla\varphi \cdot\nabla t
 \end{equation*}
 yields
  \begin{equation}
    \label{eq:4}
  -\Delta\varphi=-\frac{b}{a}|\nabla\varphi|^2+\frac{a_t}{a}\nabla\varphi\cdot\nabla t-\frac{\alpha_\varphi}{a}\frac{t^2}{\ve^2}.
  \end{equation}
From \eqref{eq:ELN-t} we deduce
\begin{equation}
  \label{eq:5}
-\Delta\left(\frac{t^2}{2}\right)=-|\nabla t|^2-t\Delta t=-|\nabla t|^2+ct|\nabla\varphi|^2-(2\alpha+\alpha_t t)\frac{t^2}{\ve^2}.
\end{equation}
 Combining \eqref{eq:4}--\eqref{eq:5} and invoking \eqref{eq:1} gives
 \be
   \label{eq:6}
   \begin{aligned}
 -\Delta\left(\frac{mt^2}{2}-\varphi\right)\leq
& (m c_2+c_3) |t| |\na\va|^2 -m  |\na t|^2+2c_4|\nabla\varphi||\nabla t|
+(c_5-m c_1)\frac{t^2}{\ve^2}.
 \end{aligned}
 \ee
We also have
\begin{equation}
  \label{eq:7}
  \begin{aligned}
  \left|\nabla\left(\frac{mt^2}{2}-\varphi\right)\right|^2=&m^2t^2|\nabla t|^2-2mt\nabla\varphi\cdot \nabla t+|\nabla\varphi|^2
  \ge 
  m^2t^2|\nabla t|^2-2m |t| |\nabla\varphi| |\nabla t|+|\nabla\varphi|^2.
  \end{aligned}
\end{equation}
 By \eqref{eq:6}--\eqref{eq:7} we obtain, for any $k>0$,
 \be
    \label{eq:8}
 \begin{aligned}
   -\Delta\left(\frac{mt^2}{2}-\varphi\right)-k \left|\nabla\left(\frac{mt^2}{2}-\varphi\right)\right|^2\leq&
   \left(mc_2|t|+c_3|t|-k\right)|\nabla\varphi|^2-\left(m+km^2t^2\right)|\nabla t|^2\\
&+\left(2c_4+2km|t|\right)|\nabla\varphi||\nabla t|+\left(c_5-mc_1\right)\frac{t^2}{\ve^2}.
 \end{aligned}
 \ee
 
Next we are looking for conditions that will insure that the right-hand side of
\eqref{eq:8} is nonpositive. First, by our assumption \eqref{eq:2}
 the last term is indeed nonpositive. The sum
of the first three terms on the right-hand side of \eqref{eq:8} is a quadratic
form in the two variables $|\nabla\varphi|,|\nabla t|$ whose discriminant $\Delta$ is given
by 
\be
\label{eq:10}
\begin{aligned}
  \Delta/4=&\left(c_4+km|t|\right)^2-\left(k-mc_2|t|-c_3|t|\right)m\left(1+kmt^2\right)\\
=&c_4^2+m(mc_2+c_3)|t|-km\left(1-2c_4|t|-m(mc_2+c_3)|t|^3\right).
\end{aligned}
\ee

By \eqref{eq:11} and \eqref{eq:10} it follows that for sufficiently large $k$ we have $\Delta\leq0$, implying that the right-hand side of \eqref{eq:8} is 
nonpositive. For such $k$ it follows that  the function 
 $v:=mt^2/2-\varphi$ satisfies
 \begin{equation*}
   \Delta(e^{kv})=ke^{kv}(\Delta v+k|\nabla v|^2)\geq 0 \text{ in }\omega.  
 \end{equation*}
  By the maximum principle, $\max_{\overline \omega} v=\max_{\partial \omega} v$, which is equivalent to
 \eqref{eq:min}.  
 
 By similar calculations, the function  
 $w:=mt^2/2+\varphi$ satisfies   $\Delta(e^{kw})\geq 0$, implying \eqref{eq:max}.
\end{proof}
\section{Asymptotic behavior of  solutions without vortices}
\label{sec:no-vortices}
In this section we shall study the asymptotic behavior of solutions   
${\ue}$ of \eqref{eq:EL} in a smooth bounded simply
connected domain $\Omega$ in $\R^2$. We assume a priori that the solutions   
are vortex-less. Actually, we shall assume a stronger condition,
namely that the
solutions are \enquote{sufficiently close}   to $\Gamma$, in a sense to be precised
below (see \eqref{eq:34}). 
 We are given a smooth boundary condition $g:\partial\Omega\to\Gamma$ of degree zero and
 a family of solutions $\{\ue\}$ of \eqref{eq:EL}. Since $g$ is of degree zero, we may globally write it as  $g=\tau(e^{\im\varphi_0})$ for some smooth $\va_0:\po\to\R$. 
 
 We next assume that 
 \begin{equation}
\label{eq:34}
  \dist(\ue (x), \Gamma)\leq\delta_1, \ \fo x\in\Omega,
\end{equation}
where $\delta_1$ is chosen to satisfy the hypotheses of \rprop{prop:max}.

Then we may write, globally in $\overline\O$ and with smooth $t_\ve$ and $\va_\ve$, 
\begin{equation}
\label{eq:36}
  \ue(x)=\tau\left(e^{\im\va_\ve(x)}\right)+t_\ve(x)\,\vec{n}\, \left(\tau\left(e^{\im \va_\ve(x)}\right)\right),\ \fo x\in\overline\O.
\end{equation}
 
Let $\zeta$ denote the harmonic extension of $\varphi_0$ to $\Omega$ and define the
$\Gamma$-valued map $u_0$ by
 \begin{equation}
\label{eq:38}
   u_0:=\tau\left(e^{\im\zeta}\right).
 \end{equation}
The main result of this section establishes, in the spirit of \cite{bbh93}, a convergence result of
$u_\varepsilon$ to the limit $u_0$. 
\begin{theorem}
  \label{th:conv-degz}
Let, for $0<\ve<\ve_0$, $\ue$ denote a solution of
\eqref{eq:EL} satisfying \eqref{eq:34}. 
Then we  have
 \begin{gather}
\ue\to u_0 \text{ in }C^{1,\alpha}(\overline \Omega)\text{ as }\ve\to 0,\, \forall\, \alpha<1, \label{eq:42-1}\\
\|\Delta\ue\|_\infty\leq C,\label{eq:42-2}\\
\|u_\varepsilon-u_0\|_\infty\leq C\varepsilon^2,\label{eq:42-3}\\
\|\nabla(u_\varepsilon-u_0)\|_\infty\leq C\varepsilon\label{eq:42-4}.
 \end{gather}
\end{theorem}

Theorem \ref{th:conv-degz} is an immediate consequence of several intermediate estimates (Lemma \ref{lem:infty} to Proposition \ref{prop:t-eps2}) that we now state and prove.

We start with two simple estimates satisfied by the solutions. These estimates are valid in any bounded domain $\O$ provided $|\ue|\le R_0$ on $\p\O$.
 \begin{lemma}
   \label{lem:infty}
  We have 
 \begin{equation}
\label{eq:linfty}
\| u_\varepsilon \|_{L^\infty(\Omega)}\leq R_0,
\end{equation}
 where $R_0$ is given by \eqref{h5}.
 \end{lemma}
 \begin{proof}
   We claim that 
   the set $E:=\{x\in\Omega;\,|\ue(x)|>R_0\}$ is empty. Indeed, this
   follows from the maximum principle for subharmonic functions since, on the one hand, we have  $|\ue|=R_0$ on $\p E$ and, on the other hand, $\ue$ satisfies in $E$
   \begin{equation*}
     \Delta (|\ue|^2)=2\left(|\nabla\ue|^2+\Delta\ue\cdot\ue)\right)\geq
     \frac{2}{\ve^2} \nabla W(\ue)\cdot\ue\geq0
   \end{equation*}
(the latter inequality following from \eqref{h5}).
 \end{proof}
 From \rlemma{lem:infty} we deduce the following gradient bound.  \begin{lemma}
   \label{lem:grad}
 We have for some constant $C$, 
 \begin{equation}
\label{eq:grad}
\| \nabla u_\ve \|_{L^\infty(\Omega)}\leq \frac{C}{\ve}.
\end{equation}
 \end{lemma}

The
 proof of Lemma \ref{lem:grad} uses the same rescaling argument as in \cite{bbh93} and is
 therefore omitted.

\smallskip
Next we prove:
\begin{lemma}
\label{lem:limt}
  We have $\lim_{\varepsilon\to 0} t_\varepsilon=0$ uniformly on $\overline{\Omega}$.
\end{lemma}
\begin{proof}
Arguing by contradiction, assume that for a subsequence $\varepsilon_n\to0$ and
a sequence of 
points $\{x_n\}\subset\Omega$ we have $\lim_{n\to\infty} t_{\varepsilon_n}(x_n)=T$
with $T\neq 0$. We distinguish two cases:
\ben
\item
 $\d\lim_{n\to\infty}\d\frac{\text{dist}(x_n,\partial\Omega)}{\varepsilon_n}=\infty$.
 \item
 $\d\liminf_{n\to\infty}\, \d\frac{\text{dist}(x_n,\partial\Omega)}{\varepsilon_n}<\infty$.
 \een

In Case 1 we define a rescaled sequence on $B_{R_n}(0)$, with
$R_n:=\d\frac{\text{dist}(x_n,\partial\Omega)}{\varepsilon_n}$, by
\begin{equation}
\label{eq:9}
  \widetilde u_{\varepsilon_n}(x):=u_{\varepsilon_n}(x_n+\varepsilon_nx).
\end{equation}

 By our assumptions,
$R_n\to\infty$ and, by standard elliptic estimates, a further subsequence,
still denoted by $\{  \widetilde u_{\varepsilon_n}\}$, converges in
$C^{1,\beta}_{\text{loc}}(\R^2)$ to a limit $\widetilde u$, solution of
$\Delta\widetilde u=\nabla W(\widetilde u)$ on all of $\R^2$ and such that $\dist (\widetilde u(x), \Gamma)\le \delta_1$, $\fo x\in\R^2$. The associated
$\widetilde t,\widetilde \varphi$ then solve the system  \eqref{eq:ELN}, with $\varepsilon=1$ on
$\R^2$ and $\widetilde t(0)=T\neq 0$. But then the proof of
\rprop{prop:max} shows that the two functions $e^{k\widetilde v}$ and
$e^{k\widetilde w}$, where 
\begin{equation*}
  \widetilde v:=\frac{m\widetilde t^2}{2}-\widetilde\varphi\text{ and } \widetilde w:=\frac{m\widetilde t^2}{2}+\widetilde\varphi,
\end{equation*}
 are subharmonic and bounded on $\R^2$. It follows that both $\widetilde v$
 and $\widetilde w$ are identically constant in $\R^2$, and therefore the
 same holds for $\widetilde t$ and $\widetilde\varphi$. In particular $\widetilde
 t\equiv T\neq 0$
 and $\nabla\widetilde \varphi\equiv 0$. But then, in view of \eqref{eq:1-1}, equation
 \eqref{eq:ELN-t} is violated. Contradiction.

 Consider next Case 2. We may assume that
$L=\d\lim_{n\to\infty}\frac{\text{dist}(x_n,\partial\Omega)}{\varepsilon_n}$ exists. By
\rlemma{lem:grad}, we have $L>0$. Arguing similarly to Case 1 we define the
rescaled sequence $\{\widetilde u_{\varepsilon_n}\}$ by \eqref{eq:9}. Again, a
subsequence converges to a solution of
$\Delta\widetilde u=\nabla W(\widetilde u)$, this time on a half-plane $H$, with a constant boundary
condition $\widetilde u=\gamma$ on $\p H$, for some point $\gamma\in
\Gamma$.

With no loss of generality, we may assume that  $H=\R\times (0, \infty)$. We know
that for some point $(x_0,y_0)\in H$ with $y_0=L$ 
 we have $\widetilde
t(x_0, L)=T\neq 0$. In addition, the boundary condition $\widetilde u=\gamma$ implies that the corresponding coordinates $\widetilde t$ and $\widetilde \va$ satisfy $\widetilde t=0$ on $\p H$ and $\widetilde \va=\Phi=$const. on $\p H$.

As above,  the functions $e^{k\widetilde v}$ and
$e^{k\widetilde w}$ are subharmonic. Since they are also bounded, the maximum
principle applies on $H$ and we obtain that both functions attain their
maximum on $\partial H$. We obtain that
\begin{equation*}
\Phi\le \widetilde\varphi(x)-\frac{m\widetilde t^2(x)}{2}\le \widetilde\varphi(x) +\frac{m\widetilde t^2(x)}{2} \le \Phi,\ \fo x\in H.
\end{equation*}
It follows that $\widetilde t\equiv 0$, contradicting $\widetilde
t(x_0)=T\neq 0$.
\end{proof}
Next we prove strong convergence of $\{\ue\}$ to $u$ in $H^1$.

\begin{proposition}
\label{prop:H1} As $\ve\to 0$, we have 
 \begin{equation}
    \label{eq:21}
\ue\to u_0 \text{ in } H^1(\Omega)\text{ and
} E_\varepsilon(\ue)\to \frac{1}{2}\int_\Omega|\nabla u_0|^2.
  \end{equation}
\end{proposition}
\begin{proof}
 Write
  $\varphi_\varepsilon=\psi_\varepsilon+\zeta$ (see \eqref{eq:38}).  The phase $\va_\ve$ is determined up to an integer multiple of $2\pi$. We fix $\va_\ve$ by imposing
  \be
  \l{hb1}
  \psi_\ve=0\text{ on }\p\O.
  \ee
  
  Note that by Corollary~\ref{rem:max} we have
  \begin{equation}
\label{eq:15}
    \|\psi_\varepsilon\|_\infty\leq 2M:=2\|\varphi_0\|_\infty.
  \end{equation}
 We rewrite \eqref{eq:ELN-phi} (dropping the subscript $\varepsilon$) as
 \begin{equation}
   \label{eq:16}
-\text{div}(a\nabla\psi)=\text{div}((a-1)\nabla\zeta)+b\left(|\nabla\psi|^2+2\nabla\psi\cdot\nabla\zeta+|\nabla\zeta|^2\right)-\frac{\alpha_\va t^2}{\ve^2}.
 \end{equation}
Multiplying \eqref{eq:16} by $\psi\in H^1_0(\O)$ and integrating 
yields
\begin{equation*}
\int_\Omega a|\nabla\psi|^2=\int_\Omega\left[(1-a)\nabla\zeta\cdot\nabla\psi+b\left(|\nabla\psi|^2+2\nabla\psi\cdot\nabla\zeta+|\nabla\zeta|^2\right)\psi-\frac{\alpha_\varphi t^2}{\varepsilon^2}\psi\right].
\end{equation*}

Using Cauchy-Schwarz inequality, \eqref{eq:15},
 \eqref{eq:1-0}, \eqref{eq:1-3}, \eqref{eq:1-5}, \rlemma{lem:limt} and
Poincar\'e inequality, it follows  
that for some constant $C=C(g)$  and for sufficiently small $\varepsilon$ we have
\begin{equation}
  \label{eq:13}
 \int_\Omega |\nabla\psi|^2\leq C\int_\Omega\frac{t^2}{\varepsilon^2}.
\end{equation}
Similarly, we rewrite \eqref{eq:ELN-t} as
\begin{equation}
  \label{eq:18}
-\Delta t+(2\alpha+\alpha_t t)\frac{t}{\varepsilon^2}=c\left(|\nabla\psi|^2+2\nabla\psi\cdot\nabla\zeta+|\nabla\zeta|^2\right).
\end{equation}
Multiplying \eqref{eq:18} by $t\in H^1_0(\O)$, integrating and using \eqref{eq:1-1}
leads to
\begin{equation}
  \label{eq:19}
\int_\Omega \left[|\nabla t|^2+(2\alpha+\alpha_t t)\frac{t^2}{\varepsilon^2}\right]=\int_\Omega  c t\left(|\nabla\psi|^2+2\nabla\psi\cdot\nabla\zeta+|\nabla\zeta|^2\right).
\end{equation}
Using \eqref{eq:1-1} and \eqref{eq:1-2}  in \eqref{eq:19} gives
\begin{equation}
  \label{eq:14}
\int_\Omega \left[|\nabla t|^2+c_1\frac{t^2}{\varepsilon^2}\right]\leq C\|t\|_\infty\left(1+\int_\Omega|\nabla\psi|^2\right).
\end{equation}
Plugging \eqref{eq:13}  into \eqref{eq:14} yields (using \rlemma{lem:limt})
\begin{equation}
  \label{eq:17}
\int_\Omega \left[|\nabla t|^2+\frac{t^2}{\varepsilon^2}\right]=o(1).
\end{equation}

Combining \eqref{eq:13} and \eqref{eq:17}, we find that
 \begin{equation}
   \label{eq:12}
  \int_\Omega |\nabla\psi|^2=o(1).
 \end{equation}

The conclusion \eqref{eq:21} clearly follows from
\eqref{eq:17}--\eqref{eq:12}.
\end{proof}
\begin{remark}
\label{rem:e0}
Note that \rprop{prop:H1} implies a uniform bound for
$E_\varepsilon(u_\varepsilon)$ for {\em all} $\varepsilon>0$. Indeed,
it suffices to consider only small values of $\varepsilon$, e.g.,
$\varepsilon<\varepsilon_0$, since for all
$\varepsilon\geq\varepsilon_0$ we deduce from the Euler-Lagrange equation \eqref{eq:EL}, \rlemma{lem:infty} and standard elliptic
estimates that
\begin{equation}
  \label{eq:37}
  \|u_\varepsilon\|_{W^{2,p}}\leq C(p,\varepsilon_0),\ \fo
  p<\infty,\ \text{ and }\ \|u_\varepsilon\|_{C^{1,\alpha}}\leq C(\alpha,\varepsilon_0),\ \fo 
  \alpha<1.
\end{equation}
We shall use this observation below for other estimates as well.
\end{remark}
\begin{lemma}
  \label{lem:w1.4}
  $\{\varphi_\varepsilon\}$ is bounded in $W^{1,4}(\Omega)$.
\end{lemma}
\begin{proof}
  We use the same notation as in the proof of \rprop{prop:H1} and
  write $\varphi_\varepsilon=\psi_\varepsilon+\zeta=\psi+\zeta$. We will actually show that
  \begin{equation}
    \label{eq:22}
 \int_\Omega |\nabla\psi_\varepsilon|^4=o(1),
  \end{equation}
 that clearly implies the result for small $\varepsilon$ (and then the
 result for any $\varepsilon>0$ follows from \rrem{rem:e0}).
 Rewrite \eqref{eq:ELN-phi} as 
 \begin{equation}
   \label{eq:23}
\left\{
\begin{aligned}
-\Delta\psi&=b|\nabla\psi|^2
     +\underbrace{2b\nabla\psi\cdot\nabla\zeta+b|\nabla\zeta|^2+\div\left((a-1)\nabla\zeta+(a-1)\nabla\psi\right)}_{R=R_\ve}\\
     &\phantom{+}\underbrace{-\frac{\alpha_\varphi t^2}{\varepsilon^2}}_{S=S_\ve}=b|\nabla\psi|^2+R+S\text{ in }\Omega,\\
    \psi&=0\text{ on }\partial\Omega.
\end{aligned}
\right.
 \end{equation}
 
 We split $\psi=\psi_1+\psi_2+\psi_3$ where 
 \begin{equation}
\label{eq:24}
\left\{
\begin{aligned}
&-\Delta\psi_1=b|\nabla\psi|^2,\ -\Delta\psi_2=R,\ -\Delta\psi_3=S\text{ in }\Omega,\\
    &\psi_1=\psi_2=\psi_3=0\text{ on }\partial\Omega.
\end{aligned}
\right.
 \end{equation}
 
Fix any $p>2$. By standard elliptic estimates, using \eqref{eq:1-0}
and \eqref{eq:1-3},
\be
\begin{aligned}
  \label{eq:25}
 \|\nabla\psi_2\|_p\le & C_1\left\{
 \|2b\nabla\psi\cdot\nabla\zeta\|_p+\|b|\nabla\zeta|^2\|_p+\|(a-1)\nabla\zeta\|_p+\|(a-1)\nabla\psi\|_p\right\}\\\le
 & C_2\|t\|_\infty\left(\|\nabla\psi\|_p+1\right).
\end{aligned}
\ee

Next we estimate $\psi_1$. Let $p>1$ and set $q:=\d\frac{2p}{p+2}$ . Then, by Sobolev
embedding (in two dimensions), $W^{2,q}(\Omega)\hookrightarrow W^{1,p}(\Omega)$. Note also
that $\d\frac{1}{2q}=\frac{1/2}{2}+\frac{1/p}{2}$, hence
\begin{equation}
  \label{eq:26}
 \|f\|_{2q}^2\leq \| f\|_2\|f\|_p,\ \fo f\in L^p(\Omega).
\end{equation}
By elliptic estimates, \eqref{eq:1-3} and \eqref{eq:26} we obtain
\be
\label{eq:27}
\begin{aligned}
  \|\nabla\psi_1\|_p\leq &C_1\|\psi_1\|_{W^{2,q}}\leq
  C_2\|b|\nabla\psi|^2\|_q\leq C_3\|t\|_\infty\|\nabla\psi\|_{2q}^2 
  \\
  \leq 
  &
  C_4\|t\|_\infty\|\nabla\psi\|_{2}\|\nabla\psi\|_{p}\leq o(1)\cdot\|t\|_\infty\|\nabla\psi\|_ p,
  \end{aligned}
  \ee
where we used \eqref{eq:12} in the last inequality.

Finally, we turn to $\psi_3$. Multiplying \eqref{eq:ELN-t} by $t$,
integrating and using \eqref{eq:1-1} and the Cauchy-Schwarz inequality yields
\begin{equation*}
\frac{c_1}{\varepsilon^2}\|t\|_2^2\leq \int_\Omega ct|\nabla\varphi|^2\leq
C\|t\|_2\|\nabla\varphi\|_4^2,
\end{equation*}
implying that (for small $\varepsilon$),
\begin{equation}
\label{eq:28}
\int_\Omega t^4\leq \int_\Omega t^2\leq C\varepsilon^4\|\nabla\varphi\|_4^4.
  \end{equation}
Recall also that by \eqref{eq:17},
\begin{equation}
  \label{eq:29}
\|t\|_2=o(\varepsilon).
\end{equation}
Again by elliptic estimates and \eqref{eq:26} we get
\begin{equation}
  \label{eq:30}
 \|\nabla\psi_3\|_p\leq \frac{C}{\varepsilon^2}\|t^2\|_q= \frac{C}{\varepsilon^2}\|t\|_{2q}^2\leq
 \frac{C}{\varepsilon^2}\|t\|_2\|t\|_p.
\end{equation}
 Choose $p=4$. Using \eqref{eq:28}--\eqref{eq:29} in
 \eqref{eq:30} gives:
 \begin{equation}
   \label{eq:31}
  \|\nabla\psi_3\|_4\leq \frac{C}{\varepsilon^2}\cdot
  o(\varepsilon)\cdot\varepsilon\|\nabla\varphi\|_4\leq o(1)\cdot\left(\|\nabla\psi\|_4+1\right). 
 \end{equation}
Combining \eqref{eq:25},\eqref{eq:27} and \eqref{eq:31} and using
\rlemma{lem:limt} we obtain
\begin{equation*}
  \| \nabla\psi\|_4\leq o(1)\cdot\left(\|\nabla\psi\|_4+1\right),
\end{equation*}
and \eqref{eq:22} follows.
\end{proof}
\begin{lemma}
\label{lem:H2}
  $\{u_\varepsilon\}$ is bounded in $H^2(\Omega)$. 
\end{lemma}
\begin{proof}
Again by \rrem{rem:e0}, it suffices to consider small $\varepsilon$.
Using the $L^4$-bound of \rlemma{lem:w1.4} for  $\nabla\varphi_\varepsilon$  in \eqref{eq:28} yields
\begin{equation}
\label{eq:33}
\int_\Omega t^2\leq C\varepsilon^4.
  \end{equation}
Since $|\nabla W(u_\varepsilon)|=O(t_\varepsilon)$, we deduce from \eqref{eq:33} that the right-hand side of the equation in \eqref{eq:EL} is bounded in $L^2(\Omega)$ and the
conclusion follows from elliptic estimates.
\end{proof}
\begin{proposition}
  \label{prop:t-eps2}
  We have
  \begin{gather}
\|t_\varepsilon\|_\infty\leq C\varepsilon^2, \label{eq:34-1}\\
\|\nabla t_\varepsilon\|_\infty\leq C\varepsilon, \label{eq:34-2}\\
\|\psi_\varepsilon\|_\infty\leq C\varepsilon^2, \label{eq:34-3}\\
\|\nabla\psi_\varepsilon\|_\infty\leq C\varepsilon. \label{eq:34-4}
  \end{gather}
\end{proposition}
\begin{proof}
  We use an argument from \cite[Step B.4]{bbh93}. Fix $q>2$. Multiplying
  \eqref{eq:ELN-t} by $|t|^{q-2}t/(\varepsilon^2)^{q-1}$ and integrating gives
  \be
    c_1\int_\Omega \left(\frac{|t|}{\varepsilon^2}\right)^q\le 
    \int_\Omega\frac{(q-1)}{\varepsilon^{2(q-1)}}|t|^{q-2}|\nabla t|^2+(2\alpha+\alpha_t
    t)\left(\frac{|t|}{\varepsilon^2}\right)^q
    =
    \int_\Omega c|\nabla\varphi|^2\left(\frac{|t|}{\varepsilon^2}\right)^{q-2}\frac{t}{\varepsilon^2}.
  \ee
  
We conclude, using H\"older inequality and \eqref{eq:1-2}, that the function $\d f_\varepsilon=f:=\frac{t}{\varepsilon^2}$ satisfies
\begin{equation*}
  c_1\|f\|_q^q\leq\int_\Omega c |\nabla\varphi|^2|f|^{q-1}\leq c_2\|\nabla\varphi\|_{2q}^2\|f\|_q^{q-1},
\end{equation*}
 i.e.,
 \begin{equation}
   \label{eq:35}
\|f\|_q\leq \frac{c_2}{c_1}\|\nabla\varphi\|_{2q}^2.
 \end{equation}
By \rlemma{lem:H2} and Sobolev embedding,  $\{\nabla\ue\}$ is uniformly
bounded in $L^r(\Omega)$ for every $r\in[1,\infty)$, and we obtain from \eqref{eq:35}
that $\|f\|_q\leq C_q$. It follows that for each $q>2$ the right-hand side of the
equation in \eqref{eq:EL} is bounded in
$L^q(\Omega)$. Hence $\{\nabla u_\varepsilon\}$ is uniformly bounded in $L^\infty(\Omega)$, 
and therefore
\begin{equation}
\label{eq:40}
 \|\nabla\varphi\|_\infty\leq \overline C,
\end{equation}
 for some constant $\overline C$. Going back to \eqref{eq:35} we obtain
 that
 \begin{equation}
   \label{eq:39}
\|f\|_q\leq \left(\frac{c_2}{c_1}\right){\overline C}^2|\Omega|^{1/q}.
 \end{equation}
Passing to the limit $q\to\infty$ in \eqref{eq:39} yields
\begin{equation*}
  \|f\|_\infty\leq  \left(\frac{c_2}{c_1}\right){\overline C}^2,
\end{equation*}
 and \eqref{eq:34-1} follows.

Next, using \eqref{eq:40} and \eqref{eq:34-1} in \eqref{eq:ELN-t}
gives the $\|\Delta t\|_\infty\leq C$. Combining this estimate with \eqref{eq:34-1}
and applying an interpolation inequality (see \cite[Lemma A.2]{bbh93})
yields \eqref{eq:34-2}. To prove
\eqref{eq:34-3}--\eqref{eq:34-4} for $\psi$, we use
\eqref{eq:40}  and the estimates
\begin{equation*}
  a-1=O(t)=O(\varepsilon^2)~\text{ and }~b=O(t)=O(\varepsilon^2),
\end{equation*}
which allow us to rewrite  \eqref{eq:23} in the form 
\bes
\Delta\psi=F+\div G,\text{ with }\|F\|_\infty=O(\varepsilon^2)\text{ and }\|G\|_\infty=O(\varepsilon^2).
\ees

The estimate
\eqref{eq:34-3} follows by elliptic estimates and finally
\eqref{eq:34-4} is deduced via interpolation as above.
\end{proof}
 The proof of the main result of this section is an easy consequence
 of our previous estimates.
\begin{proof}[Proof of \rth{th:conv-degz}]
  Since, by \eqref{ha3}, $|\nabla W(\ue)| =O(t_\varepsilon)$, \eqref{eq:42-2} follows from
  \eqref{eq:EL} and \eqref{eq:34-1}. By standard elliptic estimates we
  obtain that $\{u_\varepsilon\}$ is uniformly bounded in $C^{1,\beta}(\overline \Omega)$
  for all $\beta<1$, and \eqref{eq:42-1} follows by the Arzel\`a-Ascoli
  theorem (the identification of the limit as $u_0$ follows from
  \rprop{prop:H1}). Finally, \eqref{eq:42-3} is a consequence of
  \eqref{eq:34-1} and \eqref{eq:34-3}, while \eqref{eq:42-4} follows from
  \eqref{eq:34-2} and \eqref{eq:34-4}.
\end{proof}
\section{Boundary condition depending on $\varepsilon$}
\label{sec:dep-eps}
In the next sections we shall also need a version of
\rth{th:conv-degz} in the case where the boundary condition depends on
$\varepsilon$, and does not necessarily take  values into $\Gamma$ (analogously to \cite[Theorem 2]{bbh93} which deals with
minimizers for the GL energy). For $\Omega$ as in
\rsec{sec:no-vortices}, assume that the family $\{ g_\ve\}$ of maps 
$g_\varepsilon:\partial \Omega\to\R^2$, $\varepsilon>0$, satisfies:
\begin{align}
  \label{eq:42}
\|g_\varepsilon\|_{H^1(\partial\Omega)}&\leq C,\\
\int_{\partial\Omega} W(g_\varepsilon)&\leq C\varepsilon^2\label{eq:43}.
\end{align}

 From \eqref{eq:42}--\eqref{eq:43} it follows in particular that, possibly up to a subsequence,
 \begin{equation}
   \label{eq:44}
  g_\varepsilon\to g\text{ in }H^s (\p\O),\ \fo 0<s<1,\text{ and thus in }C^\alpha(\partial\Omega),\ \forall\,\alpha\in(0,1/2),
 \end{equation}
 for some $g\in H^{1}(\partial\Omega ; \Gamma)$. 
 
 For each $\varepsilon>0$ (or $\varepsilon\in(0,\varepsilon_0)$), let $\ue$ denote a solution of
 \begin{equation}
 \label{eq:EL-geps}
 \begin{cases}
 \d\Delta u_\ve=\frac{1}{\varepsilon^2}\nabla W(u_\ve)&\text {in }\Omega\\
 u_\ve=g_\varepsilon &\text{on }\partial\Omega
 \end{cases}.
 \end{equation}

We now make the crucial assumption that $\ue$ satisfies \eqref{eq:34} (at least for small $\ve$). Then we have 
 \be
\deg \Pi\circ g_\ve=0\text{ and thus }\deg g=0.\label{hb3}
\ee
 (Recall that $\Pi$ is the Euclidean projection on $\Gamma$.)

As before, we write $g(x)=\tau(e^{\im\varphi_0(x)})$, with $\va_0\in H^1(\p\O ; \R)$. Define, in $\Omega$, the $\Gamma$-valued map $u_0$
by \eqref{eq:38}, i.e., $u_0=\tau(e^{\im \zeta})$, where $\zeta$ is the
harmonic extension of $\varphi_0$ to $\Omega$.  Our main result establishes the
convergence of $\{\ue\}$ towards $u_0$ when $\varepsilon$ goes to zero:
 \begin{theorem}
   \label{th:conv-geps}
Under the assumptions \eqref{eq:42}--\eqref{eq:EL-geps} and \eqref{eq:34}  we have, as $\varepsilon\to0$,
  \begin{gather}
 \ue\to u_0 \text{ strongly  in }H^1(\Omega) \text{ and
   in }C^0(\overline{\Omega}),\label{eq:41-1}\\
\|\Delta u_\varepsilon\|_{L^\infty(K)}\leq C_K,\label{eq:41-2}\\
  \ue\to u_0 \text{ strongly  in }C^{1,\alpha}(K),~\forall\alpha<1, \label{eq:41-3}\\
 \|u_\varepsilon-u_0\|_{L^\infty(K)}\leq C_K\varepsilon^2\text{ and }\|\nabla(u_\varepsilon-u_0)\|_{L^\infty(K)}\leq
 C_K\varepsilon\label{eq:41-4},
  \end{gather}
 for every compact $K\subset\subset\Omega$. 
 \end{theorem}
The proof follows similar steps to those of \rsec{sec:no-vortices} and
part of the analysis carries over with slight modifications 
 to the current situation. This is the case for the analogous results to
\rlemma{lem:infty} and \rlemma{lem:limt} that we state in the next
proposition.
\begin{proposition}
  \label{prop:lim-to-one}
 We have $\|u_\varepsilon\|_\infty\leq R_0$ and $\lim_{\varepsilon\to0}t_\varepsilon=0$ uniformly on $\overline{\Omega}$. 
\end{proposition}
Next we turn to an $H^1$-convergence result, generalizing \rprop{prop:H1}.
\begin{proposition}
\label{prop:H1bis}
  We have 
 \begin{equation}
   \label{eq:32}
\ue\to u_0 \text{ in } H^1(\Omega)\text{ and
} E_\varepsilon(\ue)\to \frac{1}{2}\int_\Omega|\nabla u_0|^2.
  \end{equation}
\end{proposition}
\begin{proof}
We define the pair of functions $t_\varepsilon$ and $\varphi_\varepsilon$ associated with $u_\varepsilon$
via \eqref{eq:36}. We let $\zeta_\varepsilon$ denote the harmonic extension of
$\varphi_{\varepsilon |\partial\Omega}$ to $\Omega$. 
Analogously to the proof of \rprop{prop:H1}, we then write
  $\varphi_\varepsilon=\psi_\varepsilon+\zeta_\varepsilon$, with $\psi_\varepsilon=0$ on $\partial\Omega$.
  
   Clearly, 
  \eqref{eq:44}  implies that, possibly after subtracting  suitable integer multiples of $2\pi$ from the $\va_\ve$'s, we have $\va_{\varepsilon |\partial\Omega}\to \va_0$ in $H^{1/2}(\p\O)$, and thus 
  \begin{equation}
    \label{eq:46}
    \lim_{\varepsilon\to0} \int_\Omega |\nabla\zeta_\varepsilon|^2=\int_{\O} |\nabla\zeta|^2.
  \end{equation}
  
Repeating the calculations
  at the beginning of the proof of  \rprop{prop:H1}, with  $\zeta_\varepsilon$
  playing the role of $\zeta$, yields, analogously to \eqref{eq:13}, 
\begin{equation}
  \label{eq:41}
 \int_\Omega |\nabla\psi_\varepsilon|^2\le C\int_\Omega\frac{t_\varepsilon^2}{\varepsilon^2}.
\end{equation}
Now, since in the current setting $t_\varepsilon$ is not identically zero on
$\partial\Omega$, multiplying \eqref{eq:18} by $t_\varepsilon$, integrating and using \eqref{eq:1-1}
yields
\be
  \label{eq:45}
  \begin{aligned}
  \int_\Omega \left[|\nabla t_\ve|^2+(2\alpha+\alpha_t t_\ve)\frac{t^2_\ve}{\varepsilon^2}\right]=
  \int_{\partial\Omega} t_\varepsilon\frac{\partial t_\varepsilon}{\partial n}
  +\int_\Omega  c\, t_\ve\left(|\nabla\psi_\ve|^2+2\nabla\psi_\ve\cdot\nabla\zeta_\ve+|\nabla\zeta_\ve|^2\right)
\end{aligned},
\ee
 where $n$ stands for the outward normal on $\partial\Omega$. In order to deal
 with the boundary term in \eqref{eq:45}, we use a Pohozaev identity
 type argument, as in \cite[Proposition 3]{bbh93}. So let
 $V=(V_1,V_2)$ be a smooth vector field on $\Omega$ satisfying $V=n$ on
 $\partial\Omega$. We consider the vector field
 $V\cdot\nabla\ue=(V\cdot \na (u_\ve)_1, V\cdot \na (u_\ve)_2)$. 
 We take the scalar product of both sides of the equation in \eqref{eq:EL-geps} and $V\cdot\nabla\ue$ and
 integrate. A direct computation (see \cite{bbh93}) gives
 \be
  \label{eq:47}
 \begin{aligned}
     \int_\Omega
   (\Delta\ue)\cdot(V\cdot\nabla\ue)=&\int_\Omega\left[\frac 12\, \div V\,|\nabla\ue|^2-
(V_1)_{x_1}|(\ue)_{x_1}|^2-(V_2)_{x_2}|(\ue)_{x_2}|^2\right]\\
  &-\int_\Omega \left((V_1)_{x_2}+(V_2)_{x_1}\right)(\ue)_{x_1}\cdot(\ue)_{x_2}
  +\frac{1}{2}\int_{\partial\Omega}\left(\left|\frac{\partial\ue}{\partial n}\right|^2-\left|\frac{\partial g_\varepsilon}{\partial\sigma}\right|^2\right).
 \end{aligned}
 \ee
 (Here, $\p\sigma$ stands for the tangential derivative on $\p\O$.)
 
On the other hand, we have
\begin{equation}
\label{eq:48}
\begin{aligned}
  \frac{1}{\varepsilon^2}\int_\Omega \nabla W(\ue)\cdot (V\cdot\nabla\ue)&= \frac{1}{\varepsilon^2}\int_\Omega V\cdot \nabla\big(W(\ue)\big)
  =\frac{1}{\varepsilon^2}\left(-\int_\Omega(\div V)\, W(\ue)+ \int_{\partial\Omega} W(\ue)\right).
\end{aligned}
\end{equation}
Equating \eqref{eq:47} and \eqref{eq:48}, using \eqref{eq:42},
\eqref{eq:43}, \eqref{eq:46} and \eqref{ha3} 
yields
\begin{equation}
  \label{eq:49}
\begin{aligned}
 \int_{\partial\Omega}\left|\frac{\partial t_\varepsilon}{\partial n}\right|^2\leq &C_1\int_{\partial\Omega}\left|\frac{\partial\ue}{\partial n}\right|^2\leq
 C_2\left(1+\int_\Omega  \left[ |\nabla\ue|^2+\frac{W(\ue)}{\varepsilon^2}\right]\right)
 \leq 
 C_3\left(1+\int_\Omega \left[   |\nabla t_\varepsilon|^2+|\nabla\psi_\varepsilon|^2+\frac{t_\varepsilon^2}{\varepsilon^2}\right]\right).
\end{aligned}
\end{equation}
By \eqref{eq:49}, the Cauchy-Schwarz inequality and \eqref{eq:43} we obtain
\begin{equation}
  \label{eq:50}
\int_{\partial\Omega}t_\varepsilon\frac{\partial t_\varepsilon}{\partial n}\leq C\varepsilon \left(\int_{\partial\Omega}\left|\frac{\partial t_\varepsilon}{\partial n}\right|^2\right)^{1/2}\leq 
C'\varepsilon\left(1+\int_\Omega  \left[ |\nabla t_\varepsilon|^2+|\nabla\psi_\varepsilon|^2+\frac{t_\varepsilon^2}{\varepsilon^2}\right]\right)^{1/2}.
\end{equation}
Substituting \eqref{eq:50} in \eqref{eq:45} leads to
\begin{equation}
\label{eq:52}
\begin{aligned}
\int_\Omega \left[|\nabla t_\varepsilon|^2+c_1\frac{t_\varepsilon^2}{\varepsilon^2}\right]\leq & C\varepsilon\left(1+\int_\Omega   \left[|\nabla t_\varepsilon|^2+|\nabla\psi_\varepsilon|^2+\frac{t_\varepsilon^2}{\varepsilon^2}\right]\right)^{1/2}
+C\|t_\varepsilon\|_\infty\left(1+\int_\Omega|\nabla\psi_\varepsilon|^2\right).
\end{aligned}
\end{equation}
Combining \eqref{eq:41}, \eqref{eq:52} and \rprop{prop:lim-to-one}  we get
\begin{equation}
  \label{eq:53}
\int_\Omega \left[|\nabla t_\varepsilon|^2+\frac{t_\varepsilon^2}{\varepsilon^2}\right]=o(1).
\end{equation}
Using \eqref{eq:53} in \eqref{eq:41} finally gives 
 \begin{equation}
\label{eq:54}
  \int_\Omega |\nabla\psi_\varepsilon|^2=o(1),
 \end{equation}
and \eqref{eq:32} follows from \eqref{eq:53}--\eqref{eq:54} and \eqref{eq:46}. 
\end{proof}
Analogously to \rlemma{lem:w1.4}, and in particular to \eqref{eq:22}, we have:
\begin{lemma}
  \label{lem:w1.4-eps}
  $\psi_\varepsilon\to 0$ in $W^{1,4}(\Omega)$.
\end{lemma}
\begin{proof}
We first notice that since $\{\varphi_{\varepsilon\,|\partial\Omega}\}$ is bounded in
$H^1(\partial\Omega)$ by \eqref{eq:42}, the family $\{\zeta_\varepsilon\}$ is bounded in
$H^{3/2}(\Omega)$. Since $H^{3/2}(\Omega)\hookrightarrow W^{1,4}(\Omega)$, we get:
\begin{equation}
  \label{eq:51}
 \{\zeta_\varepsilon\}  \text{ is bounded in }W^{1,4}(\Omega).
\end{equation}
Arguing as in the proof of
\rlemma{lem:w1.4} we use \eqref{eq:23} to split
\begin{equation*}
  \psi_\varepsilon=\psi_{1,\varepsilon}+\psi_{2,\varepsilon}+\psi_{3,\varepsilon}.
\end{equation*}
 The same arguments that led to \eqref{eq:25} and \eqref{eq:27}  (with $p=4$) yield
\begin{equation}
  \label{eq:58}
 \|\nabla\psi_{2,\varepsilon}\|_4\leq C\|t_\varepsilon\|_\infty\left(\|\nabla\psi_\varepsilon\|_4+1\right)
\end{equation}
and
\begin{equation}
\label{eq:55}
  \|\nabla\psi_{1,\varepsilon}\|_4\leq C\|\psi_{1,\varepsilon}\|_{W^{2,4/3}}\leq o(1)\cdot\|t_\varepsilon\|_\infty\|\nabla\psi_\varepsilon\|_4.
  \end{equation}
  
 The only difference with respect to the case where $g_\varepsilon\equiv g$ stands in the estimate of
 $\psi_{3,\varepsilon}$. Multiplying \eqref{eq:ELN-t} by $t_\varepsilon$ and integrating
 gives
\begin{equation}
   \label{eq:56}
   \|\nabla t_\varepsilon\|_2^2+\frac{c_1}{\varepsilon^2}\|t_\varepsilon\|_2^2\leq \int_\Omega
   c\, t_\varepsilon|\nabla\varphi_\varepsilon|^2+\int_{\partial\Omega} t_\varepsilon\frac{\partial t_\varepsilon}{\partial n}\leq C\|t_\varepsilon\|_2\|\nabla\varphi_\varepsilon\|_4^2+C\varepsilon,
  \end{equation}
 where in the last inequality we used the Cauchy-Schwarz inequality
 and \eqref{eq:50} combined with \eqref{eq:53}--\eqref{eq:54}.

Next we claim that
\begin{equation}
  \label{eq:60}
\|t_\varepsilon\|_2\|\nabla\varphi_\varepsilon\|_4^2\leq \varepsilon,
\end{equation}
 for sufficiently small $\varepsilon$.
Indeed, arguing by contradiction, assume that \eqref{eq:60} does not hold,
i.e., for a sequence $\varepsilon_n\to0$ we have
\begin{equation}
\label{eq:61}
  \|t_{\varepsilon_n}\|_2\|\nabla\varphi_{\varepsilon_n}\|_4^2>\varepsilon_n.
\end{equation}
Then,  from \eqref{eq:56} we get that 
\begin{equation*}
  \frac{c_1}{\varepsilon_n^2}\|t_{\varepsilon_n}\|_2^2\leq C\|t_{\varepsilon_n}\|_2\|\nabla\varphi_{\varepsilon_n}\|_4^2,
\end{equation*}
 and the argument of the proof of \rlemma{lem:w1.4} applies, so
 thanks to \eqref{eq:51} we get, as in \eqref{eq:31}, that
 \begin{equation}
   \label{eq:57}
  \|\nabla\psi_{3,\varepsilon_n}\|_4\leq \frac{C}{\varepsilon_n^2}\cdot
  o(\varepsilon_n)\cdot\varepsilon_n\|\nabla\varphi_{\varepsilon_n}\|_4=o(1)\cdot\left(\|\nabla\psi_{\varepsilon_n}\|_4+1\right). 
 \end{equation}
 From \eqref{eq:58},\eqref{eq:55} and
 \eqref{eq:57} we obtain that \eqref{eq:22} holds, and therefore
 \begin{equation}
   \label{eq:62}
\|\nabla\varphi_{\varepsilon_n}\|_4\leq A,
 \end{equation}
for some constant $A>0$.
It follows from \eqref{eq:62} and \eqref{eq:61} that 
\begin{equation*}
  \left\|\frac{t_{\varepsilon_n}}{\varepsilon_n}\right\|_2\geq \frac{1}{A^2},
\end{equation*}
 which contradicts \eqref{eq:53}.

Using \eqref{eq:60} in \eqref{eq:56} gives
\begin{equation}
   \label{eq:64}
   \|\nabla t_\varepsilon\|_2^2+\frac{c_1}{\varepsilon^2}\|t_\varepsilon\|_2^2\leq C\varepsilon,
  \end{equation}
which implies, in particular, that
\begin{equation}
\label{eq:67}
 \int_\Omega |t_\varepsilon|^q=o(1)\int_\Omega t_\varepsilon^2 =o(\varepsilon^3),
\end{equation}
for any $q>2$.
By \eqref{eq:67}, Sobolev embedding and elliptic estimates we obtain
\begin{equation}
  \label{eq:66}
 \|\nabla\psi_{3,\varepsilon}\|_4\leq C\|\Delta\psi_{3,\varepsilon}\|_{4/3}\leq
 C\left(\int_\Omega\frac{t_\varepsilon^{8/3}}{\varepsilon^{8/3}}\right)^{3/4}\leq \frac{C}{\varepsilon^2}\cdot o(\varepsilon^{9/4})=o(\varepsilon^{1/4}). 
\end{equation}
Combining \eqref{eq:58}--\eqref{eq:55} with \eqref{eq:66} we are led
to
\begin{equation*}
  \|\nabla\psi_\varepsilon\|_4\leq o(1)\cdot\left(\|\nabla\psi_\varepsilon\|_4+1\right),
\end{equation*}
 implying that $ \|\nabla\psi_\varepsilon\|_4=o(1)$, as claimed.
\end{proof}
We next prove local estimates in $\Omega$. It suffices to consider a
sequence $\varepsilon_n\to 0$, but for simplicity we will drop the
subscript $n$. 

Fix some small $r_0>0$, depending on $\O$, such that the nearest point projection onto $\p\O$ is smooth in the set $\{ x\in\O;\, \dist (x, \p\O)<r_0\}$. Set, for $0<r<r_0$, $\Omega_r:=\{x\in \Omega;\,\dist(x,\partial\Omega)>r\}$, which is a smooth domain. Using
\eqref{eq:64} and the Fubini theorem we can find some  $r=r_\ve$
such that $r_0/2<r<r_0$ and 
\begin{equation}
  \label{eq:68}
\int_{\partial\Omega_r} \left[|\nabla t_\varepsilon|^2+\frac{t_\varepsilon^2}{\varepsilon^2}\right]\leq C\varepsilon.
\end{equation}

For such $r$, we claim the following.
\begin{lemma}
  \label{lem:eps4}
We have
  \begin{equation}
    \label{eq:69}
\int_{\Omega_r}\left[|\nabla t_\varepsilon|^2+\frac{t_\varepsilon^2}{\varepsilon^2}\right]\leq C\varepsilon^2.
  \end{equation}
 \end{lemma}
\begin{proof}
  By \eqref{eq:68} and the Cauchy-Schwarz inequality we have
  \begin{equation}
    \label{eq:59}
     \left|\int_{\partial\Omega_r}t_\varepsilon\frac{\partial t_\varepsilon}{\partial n}\right|\leq C\varepsilon^{1/2}\cdot\varepsilon^{3/2}=C\varepsilon^2.
  \end{equation}
 Similarly to the proof of \eqref{eq:56}, we multiply \eqref{eq:ELN-t}
 by $t_\varepsilon$ and 
 integrate by parts on $\Omega_r$. For the boundary integral we 
 use the improved bound \eqref{eq:59} and to bound
 $\|\varphi_\varepsilon\|_4$ we use 
 \rlemma{lem:w1.4-eps}. This yields 
 \begin{equation}
   \label{eq:65}
 \int_{\Omega_r}\left[|\nabla t_\varepsilon|^2+c_1\frac{t_\varepsilon^2}{\varepsilon^2}\right]\leq C\left(\int_{\Omega_r} t_\varepsilon^2\right)^{1/2}+C\varepsilon^2,
 \end{equation}
 which clearly implies \eqref{eq:69}.
\end{proof}
\begin{lemma}
  \label{lem:infty-grad}
  We have
  \begin{equation}
    \label{eq:70}
    \|\nabla\ue\|_{L^\infty(\Omega_r)}\leq C.
  \end{equation}
\end{lemma}
\begin{proof}
Choose $\widetilde r\in(r_0/6,r_0/5)$ satisfying \eqref{eq:68} on
$\partial\Omega_{\widetilde r}$. Then the above
arguments apply for $\Omega_{\widetilde r}$. In particular, \eqref{eq:69} holds
on $\Omega_{\widetilde r}$, and using Fubini theorem we can find $s\in(\widetilde
r_0/4, r_0/3)$ such that
\begin{equation}
   \label{eq:71}
\int_{\partial\Omega_s}\left[|\nabla t_\varepsilon|^2+\frac{t_\varepsilon^2}{\varepsilon^2}\right]\leq C\varepsilon^2.
  \end{equation}
 Since $|\nabla W(\ue)|=O(t_\varepsilon)$, the estimate \eqref{eq:69} on $\Omega_{\widetilde
   r}$ implies that $\|\Delta\ue\|_{L^2(\Omega_{\widetilde r})}=O(1)$. By standard interior
 elliptic estimates, it follows that
 \begin{equation*}
   \|\ue\|_{H^2(\Omega_s)}\leq C,
 \end{equation*}
 and then, by Sobolev embeddings,
 \begin{equation}
   \label{eq:73}
 \|\nabla\ue\|_{L^p(\Omega_s)}\leq C_p,\ \fo p\in[1,\infty).
 \end{equation}
  Next we argue similarly to the proof of \rprop{prop:t-eps2}. 
 For any $q>2$, multiplying
  \eqref{eq:ELN-t} by $|t_\varepsilon|^{q-2}t_\ve/(\varepsilon^2)^{q-1}$
  and integrating over $\Omega_s$ gives
 \be
\label{eq:63}
\begin{aligned}
    c_1\int_{\Omega_s} \left(\frac{|t_\varepsilon|}{\varepsilon^2}\right)^q\le
    &\int_{\Omega_s}\left[\frac{(q-1)}{\varepsilon^{2(q-1)}}|t_\varepsilon|^{q-2}|\nabla t_\varepsilon|^2+(2\alpha+\alpha_t
    t_\varepsilon)\left(\frac{|t_\varepsilon|}{\varepsilon^2}\right)^q\right]\\=&\int_{\Omega_s}
    c|\nabla\varphi_\varepsilon|^2\left(\frac{|t_\varepsilon|}{\varepsilon^2}\right)^{q-2}\frac{t_\varepsilon}{\varepsilon^2}+\int_{\partial\Omega_s}\left(\frac{|t_\varepsilon|}{\varepsilon^2}\right)^{q-2}
                             \left(\frac{t_\varepsilon}{\varepsilon^2}\right)\left(\frac{\partial t_\varepsilon}{\partial n}\right).
  \end{aligned}
  \ee
 We apply the above with $q=5/2$.
 Using \eqref{eq:71} and Cauchy-Schwarz inequality we estimate the boundary integral by
 \be
   \label{eq:72}
   \begin{aligned}
   \left|\int_{\partial\Omega_s}\left(\frac{|t_\varepsilon|}{\varepsilon^2}\right)^{1/2}
                             \left(\frac{t_\varepsilon}{\varepsilon^2}\right)\left(\frac{\partial
                                 t_\varepsilon}{\partial n}\right)\right|\leq &
                             \frac{1}{\varepsilon^3}\left(\int_{\p\Omega_s}|t_\varepsilon|^3\right)^{1/2}\left(\int_{\p\Omega_s}|\nabla t_\varepsilon|^2\right)^{1/2}
 =
 \frac{1}{\varepsilon^3}\cdot o(\varepsilon^2)\cdot O(\varepsilon)=o(1).
 \end{aligned}
 \ee

From \eqref{eq:63}--\eqref{eq:72} and H\"older inequality we deduce
that the function $\d f_\varepsilon:=\frac{t_\varepsilon}{\varepsilon^2}$ satisfies
\begin{equation*}
  c_1\|f_\varepsilon\|_{L^{5/2}(\Omega_s)}^{5/2}\leq\int_{\Omega_s} c |\nabla\varphi_\varepsilon|^2|f_\varepsilon|^{3/2}+o(1)\leq c_2\|\nabla\varphi_\varepsilon\|_{L^5(\Omega_s)}^2\|f_\varepsilon\|_{L^{5/2}(\Omega_s)}^{3/2}+o(1).
\end{equation*}
 Applying \eqref{eq:73} to the above yields
 \begin{equation*}
\|f_\varepsilon\|_{L^{5/2}(\Omega_s)}^{5/2}\leq C\|f_\varepsilon\|_{L^{5/2}(\Omega_s)}^{3/2}+o(1),
 \end{equation*}
implying that $\|f_\varepsilon\|_{L^{5/2}(\Omega_s)}=O(1)$ and therefore
$\|\Delta u_\varepsilon\|_{L^{5/2}(\Omega_s)}=O(1)$.
By elliptic interior estimates we obtain that $\|u_\varepsilon\|_{W^{2,5/2}(\Omega_r)}=O(1)$, and
  \eqref{eq:70} follows by Sobolev embedding.
\end{proof}
We are now ready to complete the proof of the main result of this section.
\begin{proof}[Proof of \rth{th:conv-geps}]
  The strong convergence $\ue\to u_0$ in $H^1(\Omega)$ was established in
  \rprop{prop:H1bis}. To complete the proof  of \eqref{eq:41-1} we
  need to prove the uniform convergence. This follows from the
  two uniform convergences on $\overline\Omega$:  $t_\varepsilon\to 0$  (see
  \rprop{prop:lim-to-one}) and  $\psi_\varepsilon\to 0$ (which results, by Morrey's theorem, from the
  $W^{1,4}$-convergence that was established in \rlemma{lem:w1.4-eps}).  

For the proof of  \eqref{eq:41-2} we only need to verify the following estimate:
\begin{equation}
  \label{eq:74}
  \|t_\varepsilon\|_{L^\infty(K)}\leq C_K\varepsilon^2,
\end{equation}
 for every compact $K\subset\subset\Omega$. We shall prove \eqref{eq:74} using an argument from \cite{bbh93}.  We first use
 Kato's inequality in \eqref{eq:ELN-t} to get
 \begin{equation*}
   \Delta|t_\varepsilon|\geq \sgn(t_\varepsilon)\,\Delta t_\varepsilon=(2\alpha+\alpha_tt_\varepsilon)\frac{|t_\varepsilon|}{\varepsilon^2}-c|\nabla\varphi_\varepsilon|^2\sgn(t_\varepsilon).
 \end{equation*}

Hence, by \eqref{eq:1-1} and \eqref{eq:70},
\begin{equation}
  \label{eq:75}
 -\Delta|t_\varepsilon|+c_1\frac{|t_\varepsilon|}{\varepsilon^2}\leq C_r\text{ in }\Omega_r.
\end{equation}
Now recall \cite[Lemma 2]{bbh93} that states that the radial solution $\omega=\omega(r)$
of
\begin{equation}
  \label{eq:76}
\begin{cases}
-\varepsilon^2\Delta\omega+\omega=0&\text{in }B_R(0)\\
\omega=1& \text{on }\partial B_R(0)
\end{cases}
\end{equation}
satisfies, for $\varepsilon<\d\frac{3}{4}R$,
\begin{equation}
  \label{eq:77}
  \omega(r)\leq e^{(r^2-R^2)/(4\varepsilon R)}\text{ in }B_R(0).
\end{equation}
 Let $d:=\dist(K,\partial\Omega)$, so that \eqref{eq:75} is satisfied with $r:=d/2$. Let   $x_0$ be an arbitrary point in $K$. With no loss of generality we
 may assume $x_0=0$. From
 \eqref{eq:75}--\eqref{eq:77} and the maximum principle we obtain that
 \begin{equation*}
   |t_\varepsilon|\leq C\varepsilon^2+\exp{\left[\sqrt{c_1}(|x|^2-d^2/4)/(2\ve d) \right]}\text{ in }B_{d/2}(0).
 \end{equation*}
In particular, 
\begin{equation}
\l{hi1}
  \frac{|t_\varepsilon(0)|}{\varepsilon^2}\leq  C+\frac{1}{\varepsilon^2}\, \exp{\left[-d\, \sqrt{c_1}/(8\ve) \right]}.
\end{equation}
 
 Since the right-hand side of \eqref{hi1} remains bounded as $\varepsilon\to0$,
 \eqref{eq:74} follows, completing the proof of \eqref{eq:41-2}.

From \eqref{eq:41-2} and elliptic estimates we obtain that $\ue$
 is bounded in $W^{2,p}_{\text{loc}}(\Omega)$ for every $p<\infty$, and
 \eqref{eq:41-3} follows from Morrey's theorem.
 Finally, \eqref{eq:41-4} follows from the previous estimates by the
 same arguments as in the proof of \rth{th:conv-degz}.
\end{proof}
We will need in the next section also the following variant of
\rth{th:conv-geps} and \rth{th:conv-degz}. The proof is very similar to the proofs of
these theorems, and  is therefore omitted.
\begin{theorem}
  \label{th:conv-bdry}
 Let $\Omega$ be a smooth bounded and  simply connected domain in
 $\R^2$. Let $x_0\in\partial\Omega$ and suppose that $R>0$ is sufficiently small such that 
 $\partial B_R(x_0)\cap\partial\Omega$ consists of exactly two points. 
 
 Suppose that
 $g:\partial(B_R(x_0)\cap\Omega)\to\Gamma$ is a continuous map of degree zero such that the
 restriction $g_{|\partial\Omega\cap \overline B_R(x_0)}$ is smooth. Let
 $\varphi_0$ be a continuous function such that  $g=\tau(e^{\im \varphi_0})$. Let $\zeta$ be the harmonic extension of $\va_0$ to
$\Omega\cap B_R(x_0)$ and set  $u_0:=\tau(e^{\im \zeta})$.

 For each $\varepsilon>0$ let
 $g_\varepsilon:\partial(\Omega\cap B_R(x_0)) \to\R^2$ satisfy:
\begin{align}
&g_\varepsilon=g\text{ on }\partial\Omega\cap B_R(x_0)\,\label{eq:cond1}\\
&\|g_\varepsilon\|_{H^1(\partial B_R(x_0)\cap\Omega)}\leq C,\label{eq:cond2}\\
&\int_{\partial B_R(x_0)\cap\Omega} W(g_\varepsilon)\leq C\varepsilon^2,\label{eq:cond3}\\
&
   g_\varepsilon\to g\text{ in }H^s (\partial B_R(x_0)\cap\Omega),\ 0<s<1.
\label{eq:cond4}
\end{align}

Let $u_\varepsilon$ be a solution of \eqref{eq:EL-geps} on $\Omega\cap B_R(x_0)$
(instead of $\Omega$) satisfying \eqref{eq:34}. Then for every $R_1\in(0,R)$
we have:
\begin{gather}
\|\Delta u_\varepsilon\|_{L^\infty(\overline{\Omega}\cap
  B_{R_1}(x_0))}\leq C_{R_1},\label{eq:96-1}\\
   u_\varepsilon\to u_0 \text{ in }C^{1,\alpha}(\overline{\Omega}\cap B_{R_1}(x_0)),\label{eq:96-2}\\
 \|u_\varepsilon-u_0\|_{L^\infty(\overline{\Omega}\cap
   B_{R_1}(x_0))}\leq C_{R_1}\varepsilon^2,\label{eq:96-3}\\
\|\nabla(u_\varepsilon-u_0)\|_{L^\infty(\overline{\Omega}\cap B_{R_1}(x_0))}\leq
 C_{R_1}\varepsilon\label{eq:96-4}.
\end{gather}
\end{theorem}
 Note that (possibly after passing
 to a subsequence) the condition \eqref{eq:cond4} actually follows from
 conditions \eqref{eq:cond1}--\eqref{eq:cond3}  via the compact embedding $H^1(\partial B_R(x_0)\cap\Omega)\hookrightarrow
 H^s (\partial B_R(x_0)\cap\Omega)$, $0<s<1$.

\section{General solutions}
\label{sec:nonzerodeg}

\subsection{Preliminary estimates}
\label{subsec:prelim}
 Assume that $\Omega$ is a smooth bounded domain in $\R^2$, strictly star-shaped with respect to a point $z\in\Omega$. With no loss of generality, we may assume that $z=0$, and thus 
 \be
 \l{hi2}
 x\cdot n=x\cdot n(x)\ge c>0,\ \fo x\in\p\O
 \ee
 (with $n=n(x)$ the outward normal to $\p\O$ at $x\in\p\O$).
 
  Let $g:\partial\Omega\to\Gamma$ be a smooth boundary datum of degree $d$. For each $\varepsilon>0$, let $\ue$ denote a solution of
 \eqref{eq:EL}. As in the previous sections, we denote by
 $t(x)=t_\varepsilon(x)$ the signed distance of $\ue(x)$ to $\Gamma$. In contrast with the previous sections, we do not
 impose a condition like \eqref{eq:34}, and thus we allow solutions
 with vortices.

We start with some basic estimates satisfied by the
 solutions $\ue$. We first notice that the results of
 \rlemma{lem:infty} and \rlemma{lem:grad} hold true since their proofs
 do not 
 rely on the degree of $g$.
 
 Next we prove a Pohozaev identity that does rely heavily on the
 star-shapeness assumption.
 \begin{lemma}
   \label{lem:pohozaev}
 We have
 \begin{equation}
   \label{eq:Pohozaev}
   \int_\Omega\frac{W(u_\ve)}{\ve^2}+\int_{\partial\O} \left|\frac{\partial
     u_\ve}{\partial n}\right|^2\leq C,
 \end{equation}
 for some $C$ independent of $\varepsilon$.
\end{lemma}
\begin{proof}
  The proof is standard and requires only a simple adaptation of the proof in
  \cite{bbh94}. We argue as in the
  proof of \rprop{prop:H1bis} multiplying both side of the equation in
  \eqref{eq:EL} by $V\cdot\nabla\ue$, but this time with $V=(x_1,x_2)$. For this choice of $V$, 
  \eqref{eq:47} reads
  \begin{equation}
    \label{eq:78}
  \int_\Omega \Delta\ue\cdot (V\cdot\nabla\ue)=\int_{\partial\Omega}\left[\frac{\partial\ue}{\partial n}\cdot (x\cdot\nabla\ue)-\frac{1}{2}(x\cdot n)|\nabla u_\varepsilon|^2\right],
  \end{equation}
while \eqref{eq:48} becomes
\begin{equation}
  \label{eq:79}
  \frac{1}{\varepsilon^2}\int_\Omega \nabla W(\ue)\cdot(V\cdot\nabla\ue)= \frac{1}{\varepsilon^2}\int_\Omega V\cdot \nabla(W(\ue))=-\frac{2}{\varepsilon^2}\int_\Omega W(\ue).
\end{equation}
Combining \eqref{eq:78} with \eqref{eq:79} yields
\begin{equation*}
  \frac{2}{\varepsilon^2}\int_\Omega W(\ue)+\frac{1}{2}\int_{\partial\Omega}(x\cdot n)\left|\frac{\partial\ue}{\partial n}\right|^2=\int_{\partial\Omega}\left[\frac{1}{2}(x\cdot n)\left|\frac{\partial g}{\partial\sigma}\right|^2-(x\cdot\sigma)\frac{\partial\ue}{\partial n}\cdot\frac{\partial g}{\partial\sigma}\right],
\end{equation*}
 which, in view of \eqref{hi2},  clearly implies \eqref{eq:Pohozaev}.
\end{proof}
\par Since by \eqref{eq:117} there exists $0<\mu_0<\mu$ such that
$W(\zeta)\geq \mu_0\dist^2(\zeta,\Gamma)$ for $\zeta\in B_{R_0}$, it
follows from \eqref{eq:Pohozaev} that 
\begin{equation}
  \label{eq:120}
  \int_\Omega\frac{\dist^2(u_\varepsilon(x),\Gamma)}{\varepsilon^2}\leq C.
\end{equation}

Using the two estimates \eqref{eq:Pohozaev} and \eqref{eq:grad}, we can show, using the argument of \cite[Chapter 4]{bbh94}, that for any
small $\delta_2>0$ (we will always take $\delta_2<\delta_1$, see \rprop{prop:max})
the set 
\be
\label{eq:se}
S_{\varepsilon,\delta_2}:=\{x\in \Omega;\,  \dist(u_\ve(x), \Gamma)>\delta_2\}
\ee
 can be covered by a finite number of \enquote{bad discs}
 $\{ B_{\lambda\ve}(x_j^\ve)\}_{j=1}^{k_\varepsilon}$ with
 \begin{equation}
\label{hk1}
  \{x_j^\ve\}_{j=1}^{k_\varepsilon}\subset S_{\varepsilon,\delta_2},
 \end{equation}
 where $k_\varepsilon$ is bounded  uniformly in
 $\varepsilon$. 

Indeed, we first  use \eqref{eq:grad} to choose $\lambda>0$ such that
 \begin{equation}
   \label{eq:94}
   \dist(u_\ve(x),\Gamma)>\delta_2 \Longrightarrow
   \big\{ B_{\lambda\varepsilon/4}(x)\subset\Omega\text{ and }
   \dist(u_\ve(y),\Gamma)>\delta_2/2,\,\forall y\in  B_{\lambda\varepsilon/4}(x)\big\}.
 \end{equation}
Then, we take a collection of mutually disjoint discs $\{
B_{\lambda\ve/4}(x_j^\ve)\}_{j=1}^{k_\varepsilon}$ which is maximal
with respect to the property that \eqref{hk1}
holds true.
Note that by \eqref{h1} there exists $\eta=\eta(\delta_2)$ such that
\begin{equation}
\label{eq:105}
  W(z)>\eta,\,\forall z\in B_{R_0}\setminus \Gamma_{\delta_2/2},
\end{equation}
where $\Gamma_{\delta_2/2}=\{z\in\R^2;\dist(z,\Gamma)<\delta_2/2\}$.
 Taking into account \eqref{eq:linfty} we get from \eqref{eq:94}--\eqref{eq:105}
 that
 \begin{equation}
   \label{eq:104}
   \frac{1}{\varepsilon^2}\int_{B_{\lambda\ve/4}(x_j^\ve)}W(u_\varepsilon)\geq\pi\lambda^2\eta/16,~j=1,\ldots,k_\varepsilon.
 \end{equation}
The uniform bound for $k_\varepsilon$ follows by combining \eqref{eq:104} with
\eqref{eq:Pohozaev}. By construction  
$S_{\ve,\delta_2}\subset\bigcup_{j=1}^{k_\varepsilon}B_{\lambda\ve}(x_j^\ve)$.
Next, by increasing $\lambda$ if
 necessary, we may also assume that the bad discs are well-separated, in
 the sense that $B_{4\lambda\ve}(x_j^\ve)\cap
 B_{4\lambda\ve}(x_\ell^\ve)=\emptyset$ if $j\neq \ell$ (this may results
 in decreasing the value of $k_\varepsilon$).
\par Passing to a subsequence $\varepsilon_n\to0$, but continuing to denote $\varepsilon_n$ by
  $\varepsilon$, for simplicity, we may assume $k_\varepsilon=k$ is independent of $\varepsilon$.
 Note that  outside the bad discs the function $t(x)$ is well-defined and that we have
 \begin{equation}
 \l{eq:83}
|t(x)|\le\delta_2,\ \fo x\in\O_\ve:=\O \setminus \bigcup_{j=1}^k \overline{B_{\lambda\ve}(x_j^\ve)}.
\end{equation}

The definitive value of $\delta_2$ satisfying $\delta_2\le\delta_1$ will be chosen in Section \ref{s5}; see the proof of Proposition \ref{prop:W1p}. 


 We next prove that the $x_j^\ve$'s are relatively far away from $\p\O$.
 \bl
 \l{hj2}
 We have
 \be
 \l{hj3}
 \lim_{\ve\to 0}\frac{\dist (x_j^\ve, \p\O)}{\ve}=\infty, \ j=1,\ldots, k.
 \ee
 \el
 \begin{proof}
 We argue by contradiction and assume that \eqref{hj3} does not hold for some $j$ along some sequence $\ve_n\to 0$. For notational simplicity, we drop the subscript $n$. We will obtain a contradiction via a blow up analysis. Let, for small $\ve$, $y^\ve$ denote the projection of $x_j^\ve$ onto $\p\O$, and let ${\cal R}^\ve$ denote the rotation of $\R^2$ such that ${\cal R}^\ve (0,-1)= n(y^\ve)$. Consider 
 \bes
 v_\ve(x):=\ue\left(y^\ve+\ve\, {\cal R}^\ve x  \right), \ x\in \frac 1\ve \left({\cal R}^\ve\right)^{-1}\, \left(\O- y^\ve\right).
 \ees
 
 Using \eqref{eq:linfty} and \eqref{eq:Pohozaev}, together with the boundary condition in \eqref{eq:EL}, we find that, up to a subsequence and uniformly on compacts of $H: =\{ x\in\R^2;\, x_2>0\}$, $v_\ve$ converges to a solution $v$ of
 \be
 \l{hj5}
 \begin{cases}
 \Delta w=\na W(w)&\text{in } H\\
 w=w_0\in\Gamma&\text{on }\p H\\
 \d\frac {\p w}{\p x_2}=0&\text{on }\p H
 \end{cases};
 \ee
 here, $w_0$ is a constant. Let us note that $w$ is not a constant. Indeed, we assumed by contradiction that \eqref{hj3} does not hold, and then the fact that $w$ is not constant follows from \eqref{hk1}.
 
 Consider now the map 
 \be
 \l{hj6}
 \widetilde w=\begin{cases}
 w,&\text{in }H\\
 w_0,&\text{in }\R^2\setminus H
 \end{cases}.
 \ee
 
 In view of \eqref{hj5}, the map $\widetilde w$ satisfies $\Delta \widetilde w=\na W(\widetilde w)$ in $\R^2$, first in the distributions sense, then, by elliptic regularity, in the classical sense. By unique continuation, we have $\widetilde w= w_0$. (The unique continuation property follows from \cite{muller}; there, $\widetilde w$ is a scalar function, but this is not relevant for the proof. For an explicit result relevant for vector-valued functions, see e.g. \cite[Appendix]{lopes}.) This contradicts the fact that $w$ is not a constant, and achieves the proof of the proposition.
 \end{proof}

%
 
 Now that we know that the \enquote{bad discs} $B_{\lambda\ve}(x_j^\ve)$ are well-inside $\O$, we may define the integer $d_j^\ve$ as the degrees  of $u_\ve$ on $\partial B_{\lambda\ve}(x_j^\ve)$.
  By \eqref{eq:grad}, these integers  are uniformly bounded, so we may assume that their values are independent of $\varepsilon$ as well, and thus
\be
\l{hk7}
\deg (\ue, \p B_{\lambda\ve}(x_j^\ve))=d_j,\ \fo \ve,\ j=1,\ldots, k.
\ee  

In the sequel, in case there is no risk of confusion, we shall often drop the subscript $\varepsilon$.

\smallskip
   Our next estimate yields in particular a simple answer to Open Problem 19 in the book \cite{bbh94} (previously solved in
   \cite{cm-rem} using a different method); see Corollary \ref{hk4} below.
  \begin{proposition}
  \label{prop:t} We have 
$\int_{\O_\ve} [|\nabla t|^2+{t^2}/{\ve^2}]\leq C$.
\end{proposition}
\begin{proof}
The proof uses the following pointwise inequality:
\begin{equation}
  \label{eq:112}
  |\nabla W(z)|^2\leq MW(z),\;\forall z\in B_{R_0}\,,
\end{equation}
for some $M>0$. The validity of \eqref{eq:112} for $z$ in a
neighborhood of $\Gamma$ follows from \eqref{ha3};
the extension to arbitrary $z\in B_{R_0}$ is clear (see also
Remark~\ref{rem:glaeser} below for a simple alternative argument valid
also for degenerate $W$).
Arguing as in \cite[Ch.~V]{bbh94}, we obtain using the
Galgardo-Nirenberg inequality, \eqref{eq:112} and \eqref{eq:Pohozaev} that 
\be
\begin{aligned}
\label{eq:GN}
\|\nabla u\|_{L^4(\O)}\leq &C_1\|u\|^{1/2}_{H^2}\|u\|^{1/2}_\infty \leq C_2\|u\|^{1/2}_{H^2}\leq 
 C_3{\left\{ \frac{1}{\ve^2}\|\nabla W(u)\|_2+1\right\}}^{1/2}\\
 \leq &C_4{\left\{ \frac{1}{\ve^2}\|W(u)\|_1^{1/2}+1\right\}}^{1/2}
\leq \frac{C_5}{\varepsilon^{1/2}}.
\end{aligned}
\ee

Next we multiply \eqref{eq:ELN-t} by $t$ and integrate over $\O_\ve$.
Using \eqref{eq:GN}, \eqref{eq:1-1} (recall that
$\delta_2\leq\delta_1\leq\delta_0$) and \eqref{eq:120} we get
\be
\begin{aligned}
\label{eq:om}
  \int_{\Omega_\ve}\left[ |\nabla t|^2+\frac{c_1t^2}{\ve^2}\right]\leq &
   \widetilde C_1\int_{\Omega_\ve} |\nabla\va|^2|t|+\sum_{j=1}^k \int_{\partial
     B_{\lambda\ve}(x_j^\ve)}t\frac{\partial t}{\partial n}\\\leq
   &\widetilde C_1\left(\int_{\Omega_\ve}|\nabla
     u|^4\right)^{1/2}\left(\int_{\Omega_\ve}
     t^2\right)^{1/2}+\widetilde C_2\leq 
\widetilde C_1 C_5^2
\left(\int_{\Omega_\ve}\frac{t^2}{\varepsilon^2}\right)^{1/2}+\widetilde
C_2\leq \widetilde
C_3.
 \end{aligned}
 \ee
 
  For the bound of the boundary integrals we used the estimate $\d\left|\frac{\partial t}{\partial n}\right|\leq\frac{C}{\ve}$
  on $\partial B_{\lambda\ve}(x_j^\ve)$ (by \eqref{eq:grad}). The
  conclusion of the proposition is a direct consequence of \eqref{eq:om}.
  \end{proof}
  \begin{remark}
    \label{rem:glaeser}
  The inequality \eqref{eq:112} was proved by Dieudonn\'e in
  \cite{dieud}, in connection to his simplified 
  proof to a result of Glaeser~\cite{gla} about the square root of a
  nonnegative 
  $C^2$-function. 
  A variant of Dieudonn\'e's argument, valid for any $W\in C^2(\R^2)$, goes as follows. Fix a function
  $\xi\in C^\infty_c(\R^2;[0,1])$ such that $\xi\equiv1$ on $B_{R_0}$
  and set, in $\R^2$, $F(z):=\xi(z) W(z)+(1-\xi(z))|z|^2$. Note that $F$
  is a smooth nonnegative function on $\R^2$. Let
  \begin{equation*}
    K:=\frac{1}{2}\max_{z\in\R^2}\|D^2F(z)\|,
  \end{equation*}
where $\|A\|$ stands for the spectral norm of the matrix $A$.
 By Taylor formula
 \begin{equation}
\label{eq:113}
   0\leq F(z+h)\leq F(z)+\nabla F(z)\cdot h+K|h|^2,
 \end{equation}
for every $z,h\in\R^2$. Applying \eqref{eq:113} for $\d h:=-\frac{\nabla
  F(z)}{2K}$ yields $|\nabla F(z)|^2\leq 4K F(z)$, whence \eqref{eq:112}.
  \end{remark}
  \bc
  \l{hk4}
  Let $u_\ve$ satisfy \eqref{eq:EL}. Then
  \be
  \l{hk5}
  \int_\O \left|\na \left(\dist (u_\ve, \Gamma)\right)\right|^2\le C,\ \fo \ve>0.
  \ee
  
  In particular, in the GL case, i.e., $W(u)=(1-|u|^2)^2/4$, we have
  \be
  \l{hk50}
  \int_\O \big|\na |u_\ve|\big|^2\le C, \fo \ve>0.
  \ee
  \ec
  \begin{proof}
    Since 
    \bes
    \d\|\nabla (\dist (u_\ve, \Gamma) )\|_\infty\le \|\nabla u_\ve\|_\infty \leq\frac{C}{\ve}
    \ees
    
    (by \eqref{eq:grad}), we have
    \begin{equation}
\label{eq:outom}
\int_{\bigcup_{j=1}^k B_{\lambda\ve}(x_j^\ve)}\left|\nabla \left(\dist (u_\ve, \Gamma)\right)\right|^2\leq C.
\end{equation}
The result of the corollary readily follows from Proposition \ref{prop:t} and \eqref{eq:outom}. (Recall that, in $\O_\ve$, we have $\dist (u_\ve, \Gamma)=|t_\ve|$.)

In the GL case, it suffices to note that $\dist (u_\ve, \Gamma)=1-|u_\ve|$.  
\end{proof}

\subsection{A  O($|\text{log }\ve|$) bound for the energy}
\label{subsec:log}
The main result of this section is the following.
\begin{proposition}
\label{prop:L2}
We have $E_\ve(u_\ve)\leq C(|\log\ve|+1),\ \forall\,\ve>0.$
\end{proposition}
 In view of Proposition \ref{prop:t}, of  \eqref{eq:grad} and \eqref{eq:Pohozaev}, it suffices to obtain the following  bound for the energy of the phase $\va$:
 \be
 \l{hk40}
 \int_{\O_\ve}|\na \va|^2\le C(|\log\ve|+1),\ \fo \ve>0.
 \ee
 
 Since $\va$
  is defined only locally in $\O_\ve$ (only its gradient $\nabla\varphi$ is
  defined globally), it will be convenient to introduce a new function, which is globally defined in $\O_\ve$. 
 \begin{definition}
  Let $\Pi$ denote the nearest point projection on $\Gamma$ in a $\delta_2$-tubular neighborhood of $\Gamma$.
  The $\so$-valued map 
  \bes
  \O_\ve\ni z\mapsto \tau^{-1}(\Pi(u))\cdot\prod_{j=1}^k \left(\frac{z-x_j}{|z-x_j|}\right)^{-d_j}
  \ees 
(with $d_j$ as in \eqref{hk7}) has zero degree around
   each of the holes $B_{\lambda\ve}(x_j)$, $j=1,\ldots, k$. Hence, there exists a unique (up to addition of an integer multiple of $2\pi$) scalar function $\eta=\eta_\ve$ such that
   \begin{equation}
\label{eq:84}
\tau^{-1}(\Pi(u))=e^{\im \eta}\prod_{j=1}^k \left(\frac{z-x_j}{|z-x_j|}\right)^{d_j}\text{ in }\O_\ve.
\end{equation}
\end{definition}
By adding an appropriate multiple of $2\pi$ we may assume that 
\begin{equation}
\label{eq:fix}
\min_{\partial\O}\eta_\varepsilon\in[0,2\pi).
\end{equation}
Since $g$ is smooth, we deduce from \eqref{eq:fix} that
\begin{equation}
  \label{eq:107}
 \|\eta_\varepsilon\|_{L^\infty(\partial\Omega)}\leq C(g).
\end{equation}

 Our first step consists of proving an $L^\infty$ bound for $\eta$. In order
 to be able to apply the maximum principle of \rprop{prop:max} we will remove from 
  $\O_\ve$ a collection of rays, connecting the boundaries
 of  the holes $B_{\lambda\ve}(x_j)$, $j=1,\ldots, k$,  to the boundary of $\Omega$. The choice of these \enquote{good
 rays} will depend on energy considerations.
 For any $j=1,\dots,k$ and $\alpha\in[0,2\pi)$, we let
 $D_j(\alpha)$ be the half-line 
 \bes
 D_j(\alpha):=\{ x_j^\ve +r\, e^{\im \alpha};\, r\in[\lambda\varepsilon,\infty)\},
 \ees
and then set
\bes
R_j(\alpha):=D_j\cap\Omega_\ve.
\ees


\begin{lemma}
\label{lem:alpha}
For each $j=1,\ldots,k$ and $0<\ve<1/2$, there exists $\alpha_j=\alpha_j (\ve)\in[0,2\pi)$ such that $R_j:=R_j(\alpha_j)$ satisfies
\begin{equation}
\label{eq:ray}
\int_{R_j} \left|\frac{\partial\eta}{\partial r}\right|\leq C|\log\ve|^{1/2}\|\nabla\eta\|_{L^2(\O_\ve)}.
\end{equation}
\end{lemma}

Here, $\p/\p r$ stands for the tangential derivative along $R_j$.
\begin{proof}
Since 
\begin{equation*}
\int_{\Omega_\ve} |\nabla\eta|^2\ge\int_0^{2\pi}\left(\int_{R_j(\alpha)}|\nabla\eta|^2\,rdr\right)\,d\alpha,
\end{equation*}
 there exists $\alpha_j\in[0,2\pi)$ such that
 \begin{equation}
 \label{eq:choice}
\int_{R_j(\alpha_j)}|\nabla\eta|^2\,rdr\leq \frac{1}{2\pi} \|\nabla\eta\|^2_{L^2(\Omega_\ve)}.
\end{equation}

Therefore, 
\begin{equation*}
\int_{R_j(\alpha_j)} \left|\frac{\partial\eta}{\partial r}\right|\leq \left(\int_{\lambda\ve}^{\text{diam}\, \O}\frac{dr}{r}\right)^{1/2} 
\left(\int_{R_j(\alpha_j)} \left|\frac{\partial\eta}{\partial r}\right|^2\,rdr\right)^{1/2}\leq C|\log\ve|^{1/2}\|\nabla\eta\|_{L^2(\Omega_\varepsilon)}.\qedhere
\end{equation*}
\end{proof}
 Next, we denote 
 $\omega_\ve:=\Omega_\ve\setminus\displaystyle\bigcup_{j=1}^k R_j.$
 

  For each $j$, let $\theta_j$ denote the polar coordinate around the
  point $x_j$, taking values in $[\alpha_j,\alpha_j+2\pi)$. Then the function
 \begin{equation}
 \label{eq:Theta}
\Theta=\Theta_\ve:=\sum_{j=1}^k d_j\theta_j,
\end{equation}
is  {\it smooth} in  $\omega_\ve$ and satisfies
\begin{equation}
\label{eq:bound-theta}
\|\Theta\|_{L^\infty(\omega_\ve)}\leq 4\pi \sum_{j=1}^k |d_j|.
\end{equation}
We define $\va=\va_\ve:=\eta+\Theta$ in $\omega_\ve$. Note that
\begin{equation*}
\tau^{-1}(\Pi(u))=e^{\im \eta}\prod_{j=1}^k \left(\frac{z-x_j}{|z-x_j|}\right)^{d_j}=e^{\im (\Theta+\eta)}=e^{\im \va}\text{ in }\omega_\ve,
\end{equation*}
so that $\va$ is a well-defined phase of $u$ in $\omega_\ve$. 

\begin{lemma}
\label{lem:pomega}
We have
\begin{equation}
\label{ehl1}
\|\eta\|_{L^\infty(\partial\omega_\ve)}\leq C\left( |\log\ve|^{1/2}\|\nabla\eta\|_{L^2(\O_\ve)}+1\right)
\ee
and
\be
\label{eq:bound-eta-phi}
\limsup_{\delta\to 0}\, \sup\{ |\va (x)|;\, x\in\omega_\ve,\, \dist (x, \p\omega_\ve)\le\delta\} 
\leq C\left( |\log\ve|^{1/2}\|\nabla\eta\|_{L^2(\O_\ve)}+1\right).
\end{equation}
\end{lemma}
\begin{proof}

We may assume that $0<\ve<1/2$. 
 Let $r_j(\alpha)$ be the smallest $r>\lambda\ve$ such that $x_j^\ve +r\, e^{\im \alpha}\in\p\O$. 
By \rlemma{lem:alpha} and \eqref{eq:grad}, for each $x\in [x_j^\ve +\lambda\, e^{\im \alpha}, x_j^\ve +r_j(\alpha)\, e^{\im \alpha}]$ we have
\be
\l{hl2}
\left|\eta(x)-\eta\left (x_j^\ve+r_j(\alpha_j)\, e^{\im\alpha_j}\right)\right|\le C\left( |\log\ve|^{1/2}\|\nabla\eta\|_{L^2(\O_\ve)}+1\right).
\ee
Note that \eqref{eq:grad} is needed in case
$R_j$ intersects some of the other discs
$\{B_{\lambda\varepsilon(x_l^\ve)}\}_{l\ne j}$ before reaching
$\partial\Omega$ for the first time, at $x_j^\ve+r_j(\alpha_j)$.
In particular, the following holds:
\be
\l{hl20}
\left|\eta(x_j^\ve+\lambda\ve\, e^{\im\alpha_j})-\eta\left (x_j^\ve+r_j(\alpha_j)\, e^{\im\alpha_j}\right)\right|\le C\left( |\log\ve|^{1/2}\|\nabla\eta\|_{L^2(\O_\ve)}+1\right).
\ee

 On the other hand, by 
\eqref{eq:grad} we have
\be
\l{hl3}
|\eta (x)-\eta(y)|\le C, \ j=1,\ldots, k,\ \fo x, y\in\p B_{\lambda\ve}(x_j^\ve).
\ee

 We obtain \eqref{ehl1} by combining \eqref{hl2}--\eqref{hl3} with \eqref{eq:107}. 

Finally, \eqref{eq:bound-eta-phi} follows from \eqref{ehl1} and \eqref{eq:bound-theta}.
\end{proof}

\begin{lemma}
  \label{lem:bound-eta}
  We have 
$\|\eta\|_{L^\infty(\O_\ve)}\leq C\left( |\log\ve|^{1/2}\|\nabla\eta\|_{L^2(\O_\ve)}+1\right)$.
\end{lemma}
\begin{proof}
We apply the maximum principle in \rprop{prop:max} to $\va$ on each component of the open set
  $\{ x\in\omega_\ve;\, \dist(x, \p\omega_\ve)>\delta\}$, then we let $\delta\to 0$ (with fixed $\ve$). Using \eqref{eq:bound-eta-phi}, we find that 
  \be
  \l{hl8}
  \sup_{\omega_\ve}|\va|\le C \left( |\log\ve|^{1/2}\|\nabla\eta\|_{L^2(\O_\ve)}+1\right).
  \ee
  
   The bound for $\eta$ is a consequence of \eqref{eq:bound-theta} and \eqref{hl8}. 
\end{proof}
\begin{proof}[Proof of \rprop{prop:L2}]
By \eqref{eq:ELN}, $\eta$ satisfies in $\O_\ve$
\begin{equation}
  \label{eq:eqeta}
\begin{aligned}
  -\div(a\nabla\eta)=&-\div(a\nabla\va)+\div(a\nabla\Theta)=b|\nabla\va|^2-\frac{\alpha_\va
    t^2}{\ve^2}+\div(a\nabla\Theta)
=
f+\div(a\nabla\Theta),
\end{aligned}
\end{equation}
  with
  \begin{equation}
    \label{eq:f}
    f=f_\varepsilon:=b\,|\nabla\va|^2-\frac{\alpha_\va t^2}{\ve^2}.
  \end{equation}
Above   we denoted by $\nabla\Theta$ the vector field 
\bes
\nabla\Theta=\sum_{j=1}^k
d_j\nabla\theta_j=
\sum_{j=1}^k d_j\frac{{(x-x_j^\varepsilon)}^\perp}{|x-x_j^\varepsilon|^2},
\ees
 which is smooth in $\R^2\setminus\{x_1^\ve,\ldots,x_k^\ve\}$. Here we
 used the notation $V^\perp=(-v_2,v_1)$ for a vector
 $V=(v_1,v_2)\in\R^2$. We claim that 
 \begin{equation}
   \label{eq:114}
   \|f_\ve\|_{L^1(\Omega_\varepsilon)}\leq C.
 \end{equation}
Indeed,  the second term on the
right-hand side of \eqref{eq:f} is bounded in $L^1(\O_\ve)$ by \rprop{prop:t}. The $L^1$ boundedness of the first term $b\, |\na\va|^2$  follows from the calculation \eqref{eq:om} and the inequality \eqref{eq:1-3}. 

Multiplying \eqref{eq:eqeta} by $\eta$ and integrating yields
\be
\label{eq:int-eta}
\begin{aligned}
  \int_{\O_\ve} a|\nabla\eta|^2=&\int_{\O_\ve}f\eta-\int_{\O_\ve}a\nabla\Theta\cdot\nabla\eta+\int_{\partial\O_\ve}a\, \frac{\partial\va}{\partial n}\eta\\
\leq &\|f\|_1\|\eta\|_\infty+C|\log\ve|^{1/2}\|\nabla\eta\|_{L^2(\O_\ve)}+C\|\eta\|_\infty
\leq  
C\left( |\log\ve|^{1/2}\|\nabla\eta\|_{L^2(\O_\ve)}+1\right).
\end{aligned}
\ee

The first inequality in \eqref{eq:int-eta} uses \eqref{eq:Pohozaev} on
$\partial\Omega$ and \eqref{eq:grad} on $\partial
B_{\lambda\ve}(x_j^\ve)$. The second  inequality follows from  \rlemma{lem:bound-eta}.

From \eqref{eq:int-eta} we get
\begin{equation}
  \label{eq:L2-eta}
  \|\nabla\eta\|^2_{L^2(\O_\ve)}\leq C(|\log\ve|+1),
\end{equation}
and therefore 
\be
\l{hm1}
\|\nabla\va\|^2_{L^2(\O_\ve)}\leq 2\left( \|\nabla\Theta\|^2_{L^2(\O_\ve)}+ \|\nabla\eta\|^2_{L^2(\O_\ve)}\right)\leq C(|\log\ve|+1).
\ee

 As explained above,  estimate \eqref{hm1}  implies Proposition \ref{prop:L2}.
\end{proof}

Combining \rlemma{lem:bound-eta} with \eqref{eq:L2-eta} we obtain the following.
\begin{corollary} 
\label{cor:eta}
We have
 $ \|\eta\|_{L^\infty(\Omega_\ve)}\leq C(|\log\ve|+1).$
\end{corollary}

\subsection{An $L^p$-bound for the gradient, $p\in[1,2)$}
\l{s5}
The main result of this section is
\begin{proposition}
\label{prop:W1p}
We have 
$\|\nabla u_\varepsilon\|_{L^p(\Omega)}\leq C_p$, $1\leq p<2.$ 
\end{proposition}
\begin{proof}
Fix any $p\in(1,2)$.
We will apply the bad discs construction from
Subsection~\ref{subsec:prelim} with a  $\delta_2=\delta_2(p)\le\delta_1$, that we  define below. 
By standard elliptic estimates, there exists a constant
$A_p=A_p(\Omega)$ such that the solution $w$ of the problem
\begin{equation}
  \label{eq:80}
\begin{cases}
  -\Delta w=\div {\mathbf g} &\text{ in }\Omega\\
    w=0 &\text{ on }\partial \Omega
\end{cases}
\end{equation}
 with ${\mathbf g}\in(L^p(\Omega))^2$ satisfies
 \begin{equation}
   \label{eq:81}
 \|\nabla w\|_{L^p(\Omega)}\leq A_p \|{\mathbf g}\|_{L^p(\Omega)}.
 \end{equation}
We require from $\delta_2(p)$ to satisfy
\begin{equation}
  \label{eq:82}
 0<\delta_2(p)\leq\min \left(\delta_1,\frac{1}{2c_0A_p}\right),
\end{equation}
 where $c_0$ is defined in \eqref{eq:1-0}. We choose $\delta_2=\delta_2(p)$ accordingly such that \eqref{eq:83} holds. In the sequel, $\Omega_\varepsilon$ denotes
 the set given in \eqref{eq:83} for this choice of  $\delta_2$. Note that the
 number of discs and the value of $\lambda$ may change with $\delta_2$,   but we shall use the same notation
 as before.

Let $H=H_\varepsilon$ denote the harmonic function in $\Omega$ satisfying $H=\eta_\varepsilon$ on
$\partial\Omega$. By \eqref{eq:107} and the maximum principle,
\begin{equation}
  \label{eq:100}
\|H_\varepsilon\|_{L^\infty(\Omega)}=\|\eta_\varepsilon\|_{L^\infty(\partial\Omega)}\leq C(g).
\end{equation}
Note that
\begin{equation*}
  \|\eta_\varepsilon\|_{W^{1-1/p,p}(\partial\Omega)}\leq C, 
\end{equation*}
 since
 \begin{equation*}
   \left\|\prod_{j=1}^k \left(\frac{z-x_j}{|z-x_j|}\right)^{d_j}\right\|_{W^{1,p}(\Omega)}\leq C,
 \end{equation*}
 see \eqref{eq:84}. Therefore, we also have
 \begin{equation}
   \label{eq:86}
\|H\|_{W^{1,p}(\Omega)}\leq C.
 \end{equation}

 Consider a function $\xi=\xi_\ve\in C^\infty(\overline\Omega)$ satisfying 
  \begin{equation}
  \l{hm2}
 0\leq\xi\leq 1,\    \xi\equiv 1\text{ on }\Omega\setminus\bigcup_{j=1}^k B_{2\lambda\ve}(x_j^\ve),\ 
\xi\equiv 0\text{ on }\bigcup_{j=1}^k B_{3\lambda\ve/2}(x_j^\ve),\ \|\nabla\xi\|_\infty\leq\frac{4}{\lambda\ve}.
  \end{equation}
  
Note that, by \eqref{hj3}, for small $\ve$ the discs  $\{B_{2\lambda\ve}(x_j)\}$ do not
intersect the boundary, and thus $ \xi=1$ on $\partial\Omega$. From the
properties of $\xi$ we obtain, in particular, that
\begin{equation}
\label{eq:99}
  \|\nabla\xi\|_p=\|\nabla\xi\|_{L^p(\bigcup_{j=1}^k
    B_{2\lambda\ve}(x_j^\ve))}\leq C\, \varepsilon^{2/p-1}=o(1).
\end{equation}

In $\Omega$, we set $\widetilde\eta:=\widetilde\eta_\varepsilon=\xi^2\eta$ and
$\widetilde H=\widetilde H_\varepsilon:=\xi^2H$. From
\eqref{eq:100}--\eqref{eq:99} we conclude that
\begin{equation}
  \label{eq:111}
  \|\widetilde H\|_{W^{1,p}(\Omega)}\leq C.
\end{equation}
Note that, although $\eta$ is defined only in $\O_\ve$, the function
$\widetilde\eta$ is globally defined (and smooth), since $\eta=0$ on a 
neighborhood of $\overline B_{\lambda\ve}(x_j^\ve)$. 

The function $\widetilde \eta$ satisfies
\bes
\begin{aligned}
  -\div(a\nabla \widetilde\eta)\,=&-\div(a\xi^2\nabla\eta)-\div(a\eta\nabla(\xi^2))\\
 \,=&\underbrace{-\xi^2\div(a\nabla\va)}_{F_1}\underbrace{-a\nabla(\xi^2)\cdot\nabla\va}_{F_2}+\div(\underbrace{a\xi^2\nabla\Theta}_{G_1})+\div(\underbrace{-2 a\eta\xi\nabla\xi}_{G_2})
:=&F_1+F_2+\div G_1+\div G_2.
\end{aligned}
\ees

Therefore, 
\begin{equation}
  \label{eq:85}
\begin{cases}
-\Delta (\widetilde\eta-\widetilde
H)=F_1+F_2+\div G_1+\div G_2 +\div(a\nabla \widetilde H)+\div((a-1)\nabla (\widetilde\eta-\widetilde
                      H))&\text{in }\Omega\\              
  \widetilde\eta-\widetilde H=0&\text{on }\partial\Omega
\end{cases}.
\end{equation}
By elliptic estimates, there exists $B_p=B_p(\Omega)>0$ such that the solution $w$ of the problem
\begin{equation}
  \label{eq:88}
\begin{cases}
  -\Delta w=v &\text{ in }\Omega\\
    w=0 &\text{ on }\partial \Omega
\end{cases}
\end{equation}
 with $v\in L^1(\Omega)$, satisfies
 \begin{equation}
   \label{eq:89}
 \|\nabla w\|_p\leq B_p \|v\|_1.
 \end{equation}

Note that $F_1=\xi^2f$  is bounded in
$L^1(\Omega)$; here, $f$ is defined in \eqref{eq:f}. The same holds for $F_2$ since, by \eqref{eq:grad},
\begin{equation*}
  \int_\Omega |a\nabla(\xi^2)\cdot\nabla\varphi|=\int_{\bigcup_{j=1}^k
    B_{2\lambda\ve}(x_j^\ve)\setminus
    B_{3\lambda\ve/2}(x_j^\ve)}|a\nabla(\xi^2)\cdot\nabla\varphi|\leq C_1\varepsilon^2\|\nabla\xi\|_\infty\|\nabla u\|_\infty\leq C_2.
\end{equation*}

Using the inequality
\bes
|\nabla\Theta(x)|\leq\frac{C}{r},\text{ with }r=r(x):=\dist(x,\{x_1^\ve,\ldots, x_k^\ve\}),
\ees
we find that $G_1$ is bounded in $L^p(\Omega)$. 
Similarly, $G_2$ is bounded in $L^p(\O)$, since 
\bes
\int_\O|G_2|^p\leq C_3\|\eta\|^p_\infty\left(\frac{1}{\ve}\right)^p\ve^2\leq C_4\ve^{2-p} |\log\ve|^p=o(1),
\ees

 by Corollary \ref{cor:eta} and \eqref{hm2}. Finally, $a\nabla \widetilde H$ is bounded in $L^p(\O)$ by \eqref{eq:111}.

We also note that
\begin{equation*}
  |t(x)|\leq\delta_2(p)~\text{ on }~\supp(\nabla(\widetilde\eta-\widetilde H))\subset\Omega_\varepsilon.
\end{equation*}
 Using the above in \eqref{eq:85} we get by
 \eqref{eq:81} and \eqref{eq:89} that
 \begin{align}
   \label{eq:87}
\begin{split}
 \|\nabla(\widetilde\eta-\widetilde H)\|_{L^p}&\leq A_p\left(\|(a-1)\nabla
 (\widetilde\eta-\widetilde H)\|_{L^p}+
 \|G_1\|_{L^p}+\|G_2\|_{L^p}+\|a\nabla \widetilde H\|_{L^p}\right)\\
&\phantom{=}+B_p \left(\|F_1\|_{L^1}+\|F_2\|_{L^1}\right)
\leq
A_pc_0\delta_2(p)\|\nabla(\widetilde\eta-\widetilde H)\|_{L^p}+C.
\end{split}
 \end{align}
 
 Combining \eqref{eq:82} and \eqref{eq:87}, we find that $ \|\nabla(\widetilde\eta
 -\widetilde H)\|_{L^p}\leq
 C$, which in conjunction with \eqref{eq:111} implies that
 $\|\nabla\widetilde\eta \|_{L^p}\leq C$. Since
 $\|\nabla\Theta\|_{L^p(\Omega_\varepsilon)}\leq C$, we obtain that
 \begin{equation}
 \l{hm4}
   \|\nabla u_\varepsilon\|_{L^p(\Omega\setminus\cup_{j=1}^k  B_{2\lambda\ve}(x_j^\ve))}\leq C.
 \end{equation}
 
 The conclusion of Proposition \ref{prop:W1p}   follows from \eqref{hm4} and the fact that, by \eqref{eq:grad},   $\{\nabla u_\ve\}$ is bounded in $L^p(B_{2\lambda\ve}(x_j^\ve))$.
\end{proof}

\subsection{A bound for the energy away from the singularities}

 We denote by $a_1,\ldots,a_N\in\overline{\O}$ the different limits of the families
 $\{x_j^\ve\}$, $j=1,\ldots, k$ (possibly along a subsequence). Since two different families can converge to the same limit, we have $N\le k$. 
 At this point we do not
 exclude the possibility that some of the $a_i$'s belong to $\partial\Omega$. Consider some $r>0$ such that
 \be
 \l{hm6}
 r<\min\{|a_i-a_j|;\, i\neq j\}\text{ and }r<\dist (a_j,\p\O),\ \fo j\text{ such that }a_j\in\O.
 \ee
  
  We denote
 \begin{equation*}
\widetilde\O_r:=\O\setminus \bigcup_{j=1}^N \overline {B_r(a_j)},
\end{equation*}
and by $D_j$ the  degree of $u_\ve$ on $\partial (B_\rho(a_j)\cap \O)$
for small $\ve$ and (small but fixed) $\rho$. The following equality
is clear: if $J_j:=\{ \ell;\, x_\ell^\ve\to a_j\}$, then
$D_j=\sum_{\ell\in J_j}d_\ell$. 
\begin{theorem}
\label{th:bound}
 For each $r$ as in \eqref{hm6} we have
 \begin{equation}
\label{eq:enrgy-bd}
E_\ve(u_\ve;\widetilde\O_r)\leq C(r).
\end{equation}
\end{theorem}
\begin{proof}
 By the boundedness of  $\{\nabla\eta\}$ in $L^1(\O_\ve)$ (see \rprop{prop:W1p}), it follows that there exists $\widetilde r=\widetilde r(\ve)\in (r/2,r)$ such that
 \begin{equation}
\label{eq:s}
\sum_{j=1}^N \int_{\partial B_{\widetilde r}(a_j)\cap\O} |\nabla\eta|\,d\sigma\leq C_1(r).
\end{equation}
Similarly,  we can  find for each  $j\in\{1,\ldots,N\}$ a
number $\beta_j\in[0,2\pi)$ such that the set
\bes
\widetilde R_j=\widetilde R_j(\beta_j):=\{a_j+s\, e^{\im\beta_j};\, s\ge \widetilde r\}\cap \widetilde \O_{\widetilde r}
\ees
satisfies
\begin{equation}
\label{eq:tilde}
\int_{\widetilde R_j} \left|\frac{\partial\eta}{\partial s}\right|\,ds\leq C_2(r).
\end{equation}
Repeating the proof of Lemma \ref{lem:bound-eta} and using  \eqref{eq:s} and \eqref{eq:tilde}, we find that 
\begin{equation}
\label{eq:eta-tilde}
\|\eta\|_{L^\infty(\widetilde \Omega_{\widetilde r})}\leq C_3(r).
\end{equation}
For $\varepsilon$ sufficiently small we have
\begin{equation}
  \label{eq:115}
  |x_\ell^\ve- a_j|<\widetilde r/2,~\forall\ell\in J_j,\,j=1,\ldots,N.
\end{equation}
Next, we multiply  the equation \eqref{eq:eqeta} satisfied by $\eta$ and integrate over
$\widetilde\Omega_{\widetilde r}$. 
This yields as in \eqref{eq:int-eta}
 \be
 \label{eq:ontilde}
\int_{\widetilde\Omega_{\widetilde r}}
a|\nabla\eta|^2=\int_{\widetilde\Omega_{\widetilde r}}f\eta-\int_{\widetilde\Omega_{\widetilde r}}a\nabla\Theta\cdot\nabla\eta
+\int_{\p\widetilde\O_{\widetilde r}}a\frac{\partial\va}{\partial n}\eta
:=I_1+I_2+I_3.
\ee

By \eqref{eq:114} and \eqref{eq:eta-tilde} we have $|I_1|\leq
C_4(r)$. We claim that also $|I_3|\leq C_5(r)$. Indeed, we use \eqref{eq:Pohozaev} and
\eqref{eq:eta-tilde} for the integral on $\partial
\widetilde\Omega_{\widetilde r}\cap\partial\Omega$ and for the
integral on $\partial B_{\widetilde r}(a_j)\cap\Omega$ we use \eqref{eq:s} and
the fact that thanks to \eqref{eq:115} we have
$$\left|\frac{\partial\Theta}{\partial n}\right|\leq
\frac{C}{\widetilde r}~\text{ on }\partial B_{\widetilde r}(a_j).
$$
Applying the Cauchy-Schwarz inequality to $I_2$ and the above estimates in \eqref{eq:ontilde} leads to
\begin{equation}
\label{eq:116}
  \int_{\widetilde\Omega_{\widetilde r}}
a|\nabla\eta|^2\leq C_6(r)+\int_{\widetilde\Omega_{\widetilde r}}\frac{a}{2}|\nabla\eta|^2+\int_{\widetilde\Omega_{\widetilde r}}\frac{a}{2}|\nabla\Theta|^2\,.
\end{equation}

Since $\int_{\widetilde\Omega_{\widetilde r}}(a/2)\,
|\nabla\Theta|^2\leq C_7(r)(|\log r|+1)$, we get from \eqref{eq:116}
that $\int_{\widetilde\Omega_{\widetilde r}}|\nabla\eta|^2\leq
C_8(r)$. It follows that also 
$\int_{\widetilde \Omega_{\widetilde r}}|\nabla\va|^2\leq C_9(r)$, which clearly implies \eqref{eq:enrgy-bd}.
\end{proof}
\subsection{Convergence of $\{u_{\varepsilon_n}\}$}
\label{subsec:convergence}
The bound of \rprop{prop:W1p} implies  that for a subsequence $\{
u_{\ve_n}\}$ we have 
\begin{equation}
  \label{eq:92}
  u_{\ve_n}\rightharpoonup u_*\text{ weakly in }W^{1,p}(\Omega),\ \fo p\in[1,2),
\end{equation}
 for some $u_*\in\displaystyle\bigcap_{p\in[1,2)}W^{1,p}(\Omega ;\Gamma)$.  
 The fact that $u_*$ is $\Gamma$-valued follows from \eqref{eq:92} and the
 estimate \eqref{eq:Pohozaev} that implies the convergence $t_{\varepsilon_n}\to0$
 in $L^2(\Omega)$.

 We can now further state
 \begin{proposition}
  \label{th:conv-comp}
 We have
 \begin{equation}
   \label{eq:93}
   u_{\ve_n}\to u_*\text{ in
   }C^{1,\alpha}(\overline{\Omega}\setminus\{a_1,\ldots, a_N\}),\ \fo\alpha\in(0,1).
\end{equation}
 The limit $u_*$ is a $\Gamma$-valued harmonic map in $\Omega\setminus\{a_1,\ldots, a_N\}$.  
 \end{proposition}
\begin{proof}
We  argue as in \cite[Proof of Theorem~VI.1]{bbh94}. 
 For notational simplicity, we drop in what follows the subscript $n$.
It suffices to show that for every $x_0\in\overline{\Omega}\setminus\{a_1,\ldots, a_N\}$ there
exists $R>0$ such that $\ue\to u_*$ in $C^{1,\alpha}(\overline{\Omega}\cap B_R(x_0))$.
Consider first the case $x_0\in\Omega$. We choose $R>0$ such that
$B_{2R}(x_0)\Subset {\Omega}\setminus\{a_1,\ldots, a_N\}$. Since, by \eqref{eq:enrgy-bd},
\begin{equation*}
  E_{\varepsilon}(\ue;B_{2R}(x_0))\leq C,
\end{equation*}
 we can use Fubini's theorem to find $R'\in(R,2R)$ such that (after
 passing to a further subsequence),
 \begin{equation*}
   \int_{\partial B_{R'}(x_0)}\left[\frac{1}{2}|\nabla\ue|^2+\frac{W(\ue)}{\varepsilon^2}\right]\leq C.
 \end{equation*}
Then, applying \rth{th:conv-geps} we obtain that, up to a further subsequence, $\ue\to u_0$ in
$C^{1,\alpha}(\overline B_R(x_0))$, and that $u_0$ is a harmonic map in $B_R(x_0)$
(since it can be written as $u_0=\tau(e^{\im \zeta})$ where $\zeta$ is a harmonic
function in $B_R(x_0)$). Using the uniqueness of the limit, we
find that $u_0=u_*$, and that the original subsequence $\{ u_{\ve_n}\}$ converges to $u_*$ in $C^{1,\alpha}(\overline B_R(x_0))$.
 
It remains to consider the case $x_0\in\partial\Omega\setminus\{a_1,\ldots, a_N\}$ (at this stage we do
not exclude the possibility that   some of the $a_j$'s belong to
$\partial\Omega$). We choose a small $R>0$ such $\partial B_R(x_0)\cap\partial\Omega$ consists of
exactly two points and  
\begin{equation*}
  R<\min_{1\leq j\leq N} |x_0-a_j|.
\end{equation*}
Again by  \eqref{eq:enrgy-bd}, we have
\begin{equation*}
  E_{\varepsilon}(\ue;\Omega\cap B_{2R}(x_0))\leq C,
\end{equation*}
 and by Fubini's theorem there exists $R'\in(R,2R)$ such that (after
 passing to a further subsequence),
 \begin{equation*}
   \int_{\partial B_{R'}(x_0)\cap\Omega}\left[\frac{1}{2}|\nabla\ue|^2+\frac{W(\ue)}{\varepsilon^2}\right]\leq C.
 \end{equation*}
Applying \rth{th:conv-bdry} we obtain that $\ue\to u_*$ in
$C^{1,\alpha}(B_R(x_0)\cap\overline{\Omega})$.
\end{proof}
Next we deduce further properties of the map $u_*$ that will enable us
to conclude the proof of \rth{th:main}.
\begin{proposition}
  We have $\{a_1,\ldots, a_N\}\subset\Omega$.
\end{proposition}
\begin{proof}
  The proof is the same as that of \cite[Theorem X.4]{bbh94}, so we
  just mention the main idea. By Pohozaev identity \eqref{eq:Pohozaev}
  and Proposition \ref{th:conv-comp} it follows that
  \begin{equation}
    \label{eq:95}
   \int_{\partial\Omega} \left|\frac{\partial u_*}{\partial n}\right|^2<\infty.
  \end{equation}
 The map $v_*:=\tau^{-1}\circ u_*$ is an $\so$-valued smooth harmonic map on
 $\overline{\Omega}\setminus\{a_1,\ldots, a_N\}$, and  satisfies: $v_*\in W^{1,p}(\Omega ; \so)$ for all
 $p\in[1,2)$, $v_*=\tau^{-1}\circ g$ on $\partial\Omega$, and thanks to \eqref{eq:95}, also
 \begin{equation*}
   \int_{\partial\Omega} \left|\frac{\partial v_*}{\partial n}\right|^2<\infty.
  \end{equation*}
  
 Therefore, all the hypotheses of \cite[Lemma X.14]{bbh94} are satisfied,
 and we can conclude that $v_*$ is smooth in a neighborhood of
 $\partial\Omega$. Clearly, the same holds for $u_*$.
\end{proof}
To conclude the proof of \rth{th:main} we need to show that the limit
$u_*$ has the form given in \eqref{eq:91}.
\begin{proposition}
\l{p515}
 We have
\begin{equation}
\label{eq:96}
  u_*(z)=\tau\left(e^{\im \eta(z)}\left(\frac{z-a_1}{|z-a_1|}\right)^{D_1}\cdots\left(\frac{z-a_N}{|z-a_N|}\right)^{D_N}\right), 
\end{equation}
 for some smooth harmonic function $\eta$ in $\overline \Omega$  and $D_1,\ldots,D_N\in\Z\setminus\{0\}$.
\end{proposition}
 
 Equivalently, Proposition \ref{p515} asserts that 
the $\so$-valued harmonic map $\tau^{-1}\circ u_*$ is the canonical
harmonic map associated with $\tau^{-1}\circ g$ and $\{(a_j,D_j)\}_{j=1}^N$, as
defined in \cite[Sec~I.3]{bbh94}.
\begin{proof}
  We apply the same argument as in \cite[Ch.~VII]{bbh94}, which uses
  the Hopf differential. Setting
  \begin{equation*}
    \omega=\omega_\ve:=\left|(\ue)_{x_1}\right|^2-\left|(\ue)_{x_2}\right|^2-2\imath\, (\ue)_{x_1}\cdot (\ue)_{x_2},
  \end{equation*}
 we find by a direct computation
 \begin{equation}
   \label{eq:97}
   \frac{\partial\omega}{\partial\overline z}=\frac{1}{2}\left(
   \omega_{x_1}+\imath\, \omega_{x_2}\right)=\Delta\ue\cdot\left((\ue)_{x_1}-\imath\, (\ue)_{x_2}\right)=2\Delta\ue\cdot\frac{\partial\ue}{\partial z}.
 \end{equation}
 
Moreover, by \eqref{eq:EL},
\begin{equation}
  \label{eq:98}
 \frac{\partial }{\partial z}W(\ue)=\nabla W(\ue)\cdot\frac{\partial\ue}{\partial z}=\varepsilon^2\Delta\ue\cdot
 \frac{\partial\ue}{\partial z}.
\end{equation}

By \eqref{eq:97}--\eqref{eq:98},
\begin{equation*}
  \frac{\partial\omega}{\partial\overline z}=\frac{\partial}{\partial z}\left(\frac{2W(\ue)}{\varepsilon^2}\right).
\end{equation*}
 Note that up to a further subsequence we have
 \begin{equation}
   \label{eq:109}
  \frac{W(\ue)}{\varepsilon^2}\stackrel{*}{\rightharpoonup}\sum_{j=1}^N m_j\delta_{a_j},
 \end{equation}
 for some positive  $m_j$'s. (Convergence is in the weak
 star topology of $C(\overline\Omega)$.) Indeed, combining  
 \eqref{eq:41-4} and \eqref{eq:96-3} we obtain, for any sufficiently
 small $R>0$,
 \begin{equation*}
   W(\ue)\leq C_R\varepsilon^2\text{ in }\overline\Omega\setminus\bigcup_{j=1}^N B_R(a_j),
 \end{equation*}
 which clearly implies \eqref{eq:109} with $m_j\ge 0$. The fact that $m_j>0$ for all
 $j$ follows from \eqref{eq:104}. 
 
 Defining the distribution
 \begin{equation*}
   \alpha=\alpha_\ve:=\frac{\partial}{\partial z}\left[\left(\frac{1}{\pi z}\right)\ast\left(\chi_\Omega\frac{W(\ue)}{\varepsilon^2}\right)\right],
 \end{equation*}
 we obtain by a direct calculation that  $\beta=\beta_\ve:=\omega-2\alpha$ is a
 holomorphic function in $\Omega$ (see also \cite{bbh94}). 
 
 Since, by \eqref{eq:41-2}--\eqref{eq:41-3},
 \begin{equation*}
   \left\|\frac{1}{\varepsilon^2}\nabla W(\ue)\frac{\partial\ue}{\partial x_j}\right\|_{L^\infty(\overline\Omega\setminus\bigcup_{j=1}^N
     B_R(a_j))}\leq C,\  j=1,2,
 \end{equation*}
we obtain that $\d\left\{\frac{W(\ue)}{\varepsilon^2}\right\}$ is bounded in $C^1(\overline\Omega\setminus\bigcup_{j=1}^N B_R(a_j))$,
 and we deduce by the argument of
 \cite{bbh94} that $\{\beta_\ve\}$ is bounded in
 $C^0_{\text{loc}}(\Omega)$. It follows that, up to a further subsequence, $\beta_\ve\to\beta_\ast$ in $C^k_{\text{loc}}(\Omega)$,
 $\forall\, k$, for some holomorphic function $\beta_*$ in
 $\Omega$. In addition, using \eqref{eq:109} we find that
 \begin{equation}
\label{eq:102}
   \alpha_\ve\to\alpha_\ast:=\frac{\partial}{\partial z}\left(\frac{1}{\pi
     z}\right)\ast\sum_{j=1}^N m_j\delta_{a_j} =-\frac 1\pi\, \sum_{j=1}^N m_j\, \pv\frac 1{(z-a_j)^2}\text{ in }{\mathcal D}'(\R^2).
 \end{equation}
 
Therefore, $\omega_\ve=\beta_\ve+2\alpha_\ve\to \beta_\ast+2\alpha_\ast$ in
${\mathcal D}'(\Omega)$.

 Since Proposition \ref{th:conv-comp} implies that
\bes
\omega_\ve
\to \omega_*:=\left|(u_*)_{x_1}\right|^2-\left|(u_*)_{x_2}\right|^2-2\imath\, (u_*)_{x_1}\cdot (u_*)_{x_2}\text{ in }C^0_{\text{loc}}(\overline\Omega\setminus \{a_1,\ldots,a_N\}),
\ees
we obtain 
 \begin{equation}
   \label{eq:103}
\omega_*=\beta_*+2\alpha_*\text{ in } {\mathcal D}'(\Omega\setminus \{a_1,\ldots,a_N\}).
 \end{equation}

Fix any $j\in\{1,\ldots,N\}$ and assume without loss of generality that $a_j=0$. Recall that $\tau^{-1}\circ u$ is a harmonic map in $\O\setminus \{ a_1,\ldots, a_N\}$ (Proposition \ref{th:conv-comp}) and belongs to $W^{1,p}(\O)$ when $1\le p<2$ (Proposition \ref{prop:W1p}). In addition, we have $\deg (u_*, 0)=D_j$.  Arguing as in \cite[Remark I.1]{bbh94}
we may  write, near $0$,   
\begin{equation*}
  u_*=\tau\left(\exp(\im D_j\theta+\im c_j\log r+\im h)\right),
\end{equation*}
 where $h$ is a harmonic function.
 
It follows that if we write, locally near $0$, $u_*=\tau(e^{\im \varphi})$ with
$\varphi:=D_j\theta+c_j\log r+h$, then we have
\begin{equation}
\label{eq:101}
  \omega_*=\left|(u_*)_{x_1}\right|^2-\left|(u_*)_{x_2}\right|^2-2\imath\, (u_*)_{x_1}\cdot (u_*)_{x_2}=\left(\varphi_{x_1}-\imath\, \varphi_{x_2}\right)^2=\left(\frac{c_j-\imath\, D_j}{z}+2\frac{\p h}{\p z}\right)^2.
\end{equation}
 
 From \eqref{eq:102}--\eqref{eq:101} we
 obtain
 \begin{equation*}
   (c_j-\imath\, D_j)^2=-\frac{2\, m_j}{\pi},
 \end{equation*}
 implying that  $c_j=0$  and also $\d m_j=\frac{\pi\, D_j^2}{2}$. The fact
 that $c_j=0$ for all $j$ implies that $u_*$ has the form
 \eqref{eq:96}. Since we know already that $m_j\neq0$ for all $j$, it
 follows that 
 also $D_j\neq0$ for
all $j$. 
\end{proof}
\begin{remark}
  \label{rem:ren} Arguing as in \cite[Ch.~VII]{bbh94}, we may conclude
  from \eqref{eq:101} that $\p h/\p z=0$. This implies that the
  configuration $(a_1,\ldots,a_N)$ is a critical point of the
  renormalized energy associated with the degrees $(D_j)_{j=1}^N$ and
  the $\so$-valued boundary condition $\tau^{-1}\circ g$, see
  \cite[Corollary VIII.1]{bbh94}.
\end{remark}
\bibliographystyle{chicago}

\begin{thebibliography}{6}

\bibitem{almeida-beth98}
Almeida, L., Bethuel, F. (1998).
\newblock Topological methods for the {G}inzburg-{L}andau equations.
\newblock {\em J. Math. Pures Appl. (9)\/}~{\em 77\/}(1), 1--49.

\bibitem{as03}
Andr{\'e}, N., Shafrir, I. (2003).
\newblock On a singular perturbation problem involving the distance to a curve.
\newblock {\em J. Anal. Math.\/}~{\em 90}, 337--396.

\bibitem{ansh07}
Andr{\'e}, N., Shafrir, I. (2007).
\newblock On a singular perturbation problem involving a ``circular-well''
  potential.
\newblock {\em Trans. Amer. Math. Soc.\/}~{\em 359\/}(10), 4729--4756
  (electronic).

\bibitem{bbh93}
Bethuel, F., Brezis, H., H{\'e}lein, F. (1993).
\newblock {Asymptotics for the minimization of a Ginzburg-Landau functional}.
\newblock {\em Calc. Var. and Partial Differential Equations\/}~{\em 1},
  123--148.

\bibitem{bbh94}
Bethuel, F., Brezis, H., H{\'e}lein, F. (1994).
\newblock {\em Ginzburg-Landau Vortices}.
\newblock Birkh{\"a}user.

\bibitem{cm-rem}
Comte, M., Mironescu, P.  (1996).
\newblock Remarks on nonminimizing solutions of a {G}inzburg-{L}andau type
  equation.
\newblock {\em Asymptotic Anal.\/}~{\em 13\/}(2), 199--215.


\bibitem {dieud}
Dieudonn\'e, J. (1970).
\newblock Sur un th\'eor\`eme de {G}laeser.
\newblock {\em J. Analyse Math.\/}~{\em 23}, 85--88.


\bibitem{gitr01}
Gilbarg, D., Trudinger, N.S. (2001).
\newblock {\em Elliptic partial differential equations of second order}.
\newblock Classics in Mathematics. Berlin: Springer-Verlag.

\bibitem{gla}
Glaeser, G. (1963).
\newblock Racine carr\'ee d'une fonction diff\'erentiable.
\newblock {\em Ann. Inst. Fourier (Grenoble)\/}~{\em 13\/}(fasc. 2), 203--210.

\bibitem{lin95}
Lin, F.H. (1995).
\newblock Mixed vortex-antivortex solutions of {G}inzburg-{L}andau equations.
\newblock {\em Arch. Rational Mech. Anal.\/}~{\em 133\/}(2), 103--127.

\bibitem{lopes}
Lopes, O. (1996).
\newblock Radial symmetry of minimizers for some translation and rotation
  invariant functionals.
\newblock {\em J. Differential Equations\/}~{\em 124\/}(2), 378--388.

\bibitem{muller}
M\"uller, C. (1954).
\newblock On the behavior of the solutions of the differential equation
  {$\Delta U=F(x,U)$} in the neighborhood of a point.
\newblock {\em Comm. Pure Appl. Math.\/}~{\em 7}, 505--515.

\bibitem{zhou-qing99}
Zhou, F., Zhou, Q. (1999).
\newblock A remark on multiplicity of solutions for the {G}inzburg-{L}andau
  equation.
\newblock {\em Ann. Inst. H. Poincar\'e Anal. Non Lin\'eaire\/}~{\em 16\/}(2),
  255--267.


\end{thebibliography}

\end{document}